\documentclass[3p]{elsarticle}

\usepackage{amssymb}

\journal{}

\usepackage{amsthm}
\usepackage{mathrsfs}
\usepackage{array}
\usepackage{amsmath}
\usepackage{amsfonts}
\usepackage{dsfont}
\usepackage{latexsym}
\usepackage{graphicx}
\usepackage{epstopdf}
\usepackage{color}
\usepackage{booktabs}
\usepackage[colorlinks]{hyperref}
\usepackage[T1]{fontenc}

\newtheorem{theorem}{Theorem}[section]
\newtheorem{lemma}[theorem]{Lemma}

\newtheorem{problem}[theorem]{Problem}

\begin{document}

\begin{frontmatter}

\title{{\bf Cubic bipartite graphs with minimum spectral gap}}
\author[a]{Ruifang Liu}
\author[a]{Jie Xue\corref{cor1}}
\address[a]{School of Mathematics and Statistics, Zhengzhou University, 450001 Zhengzhou, China}
\cortext[cor1]{Corresponding author. Email addresses: rfliu@zzu.edu.cn, jie\_xue@126.com.}

\begin{abstract}
The difference between the two largest eigenvalues of the adjacency matrix of a graph $G$ is called
the spectral gap of $G.$ If $G$ is a regular graph, then its spectral gap is equal to algebraic connectivity.
Abdi, Ghorbani and Imrich, in [European J. Combin. 95 (2021) 103328], showed that the minimum algebraic connectivity of cubic connected graphs on $2n$ vertices is $(1+o(1))\frac{\pi^{2}}{2n^{2}}$, which is attained on non-bipartite graphs. Motivated by the above result,
we in this paper investigate the algebraic connectivity of cubic bipartite graphs.
We prove that the minimum algebraic connectivity of cubic bipartite graphs on $2n$ vertices is $(1+o(1))\frac{\pi^{2}}{n^{2}}$.
Moreover, the unique cubic bipartite graph with minimum algebraic connectivity is completed characterized.
Based on the relation between the algebraic connectivity and spectral gap of regular graphs,
the cubic bipartite graph with minimum spectral gap and the corresponding asymptotic value are also presented.
In [J. Graph Theory 99 (2022) 671--690], Horak and Kim established a sharp upper bound for the number of perfect matchings in terms of the Fibonacci number.
We obtain a spectral characterization for the extremal graphs by showing that a cubic bipartite graph has the maximum number of perfect matchings
if and only if it minimizes the algebraic connectivity.
\end{abstract}


\end{frontmatter}

\section{Introduction}
\noindent All graphs considered in this paper are simple, connected and undirected. Given a graph $G$, its Laplacian matrix is defiened as
$L(G)=D(G)-A(G)$, where $A(G)$ and $D(G)$ are the adjacency matrix and the diagonal matrix of vertex degrees of $G$, respectively.
The Laplacian matrix is also known as the Kirchhoff matrix.
Research on the Laplacian matrix can be traced back to the famous Matrix-tree Theorem \cite{Kirchhoff1847}.
Readers can find more results on the Laplacian matrix in survey papers \cite{Merris1994,Merris1995} by Merris and \cite{Mohar1992} by Mohar.

The Laplacian matrix is symmetric, positive semidefinite, and its row sum is zero.
The eigenvalues of $L(G)$ are called the Laplacian eigenvalues of $G$.
The Laplacian eigenvlaues are related to some structural properties of a graph (see, e.g., \cite{Cardoso2007,Gu2022,Hedetniemi2016,Higuchi2009,Wu2014,Zhang2003}).
The second smallest Laplacian eigenvalue is popularly known as the algebraic connectivity of $G,$ and is usually denoted by $a(G)$. The difference between the two largest eigenvalues of the adjacency matrix of a graph $G$ is
called the spectral gap of $G.$ If $G$ is a regular graph, then its spectral gap is equal to the algebraic connectivity.
The algebraic connectivity is an important parameter in spectral graph theory, and has received much attention (see, e.g. \cite{deAbreu2007,Guo2011,Li2010,Wang2010,Zhang2017ELA}).
Fiedler \cite{Fiedler1973} proved that a graph is connected if and only if its algebraic connectivity is positive.
Moreover, the algebraic connectivity provides a lower bound of the vertex (edge) connectivity of a graph \cite{Fiedler1973,Kirkland2002}.
On the other hand, the algebraic connectivity has important applications in other research fields as well.
As an example, in consensus problems, the algebraic connectivity plays a crucial role in convergence analysis of consensus and alignment algorithms,
and it quantifies the speed of convergence of consensus algorithms \cite{Olfati2004,Ogiwara2017}.

A classical topic on the algebraic connectivity is to determine the minimum (maximum) algebraic connectivity of graphs under some constraints.
For trees, it is easy to see that the minimum and maximum algebraic connectivity are attained on the path and the star, respectively.
More researches on the algebraic connectivity of trees were seen in \cite{Grone1990,Kolokolnikov2015,Shao2008}.
The minimum algebraic connectivity of graphs with cycles is considered in many papers.
In \cite{Fallat1998}, Fallat and Kirkland conjectured that the lollipop graph is the unique graph with minimum algebraic connectivity among all graphs with given girth.
This conjecture was verified in \cite{Fallat2002} for special graphs and was completely confirmed by Guo \cite{Guo2008}.
The minimum algebraic connectivity of Hamiltonian graphs was determined by Guo, Zhang and Yu \cite{Guo2018}.
More generally, Xue, Lin and Shu \cite{Xue2019} studied minimum algebraic connectivity of graphs with given circumference.
It is worth mentioning that the eigenvector plays an important role in characterizing graphs with minimum (maximum) algebraic connectivity.
An eigenvector corresponding to the algebraic connectivity is called a Fiedler vector.
Some important properties for the Fiedler vector were established by Fiedler \cite{Fiedler1975} and Kirkland, Rocha and Trevisan \cite{Kirkland2015}, respectively.

A graph is $k$-regular if the degree of each vertex is equal to $k$.
In general, 3-regular graphs are also called cubic graphs, which play a prominent role in graph theory.
During recent years, many researches concerning cubic graphs were reported (see, e.g., in \cite{Horak2022,Macajova2021A,Macajova2021B,Das2021,Knauer2019,Cames2022}).

The cubic graph with minimum algebraic connectivity (spectral gap) was discussed in \cite{Brand2007,Guiduli1997}.
The unique extremal graph was determined by Brand, Guiduli and Imrich \cite{Brand2007}.
Recently, Abdi, Ghorbani and Imrich \cite{Abdi2021} obtained the asymptotic value of minimum algebraic connectivity of cubic connected graphs.
Very recently, we are happy to see that Abdi and Ghorbani \cite{Abdi2022} determined the structure of
connected quartic graphs with minimum algebraic connectivity (spectral gap) and the corresponding asymptotic value.

\begin{theorem}{\rm (\cite{Abdi2021})}
  The minimum algebraic connectivity of connected cubic graphs on $2n$ vertices is $(1+o(1))\frac{\pi^{2}}{2n^{2}}$.
\end{theorem}

We remark that, among cubic connected graphs, the minimum algebraic connectivity is attained on a non-bipartite cubic graph.
Hence it is very interesting to consider the following problem.

\begin{problem}
What is the minimum algebraic connectivity of connected cubic bipartite graphs? Moreover, characterize the unique extremal graph.
\end{problem}

In this paper, we focus on the above problem and prove Theorems \ref{the_min_val} and \ref{the_extr}, which indicate that the minimum algebraic connectivity of cubic bipartite graphs is twice that of cubic graphs.

\begin{theorem}\label{the_min_val}
The minimum algebraic connectivity of connected cubic bipartite graphs on $2n$ vertices is $(1+o(1))\frac{\pi^{2}}{n^{2}}$.
\end{theorem}

\begin{figure}[ht]
  \centering
  \includegraphics[scale=0.7]{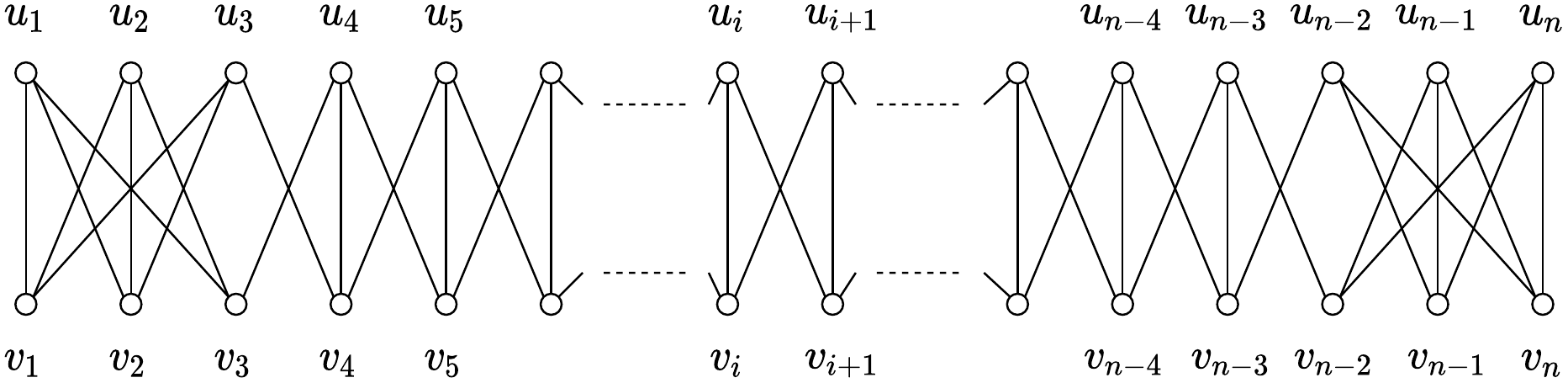}
  \caption{A cubic bipartite graph $H_{2n}.$}
  \label{fig_main}
  \end{figure}

In order to prove the asymptotic value of the minimum algebraic connectivity, we need to know the structure of the extremal graph.
For $n\geq 6$, let $H_{2n}$ be a cubic bipartite graph on $2n$ vertices, as shown in Figure \ref{fig_main}.
We show that $H_{2n}$ is the unique extremal graph.

\begin{theorem}\label{the_extr}
  Among all connected cubic bipartite graphs, $H_{2n}$ is the unique graph with the minimum algebraic connectivity, where $n\geq 6.$
\end{theorem}

A matching in a graph is a set of edges such that no two of which have a common vertex.
A matching is perfect if every vertex of the graph is incident to an edge of the matching.
In \cite{Horak2022}, Horak and Kim studied the number of perfect matchings in cubic graphs.
For cubic bipartite graphs, they provided the maximum number of perfect matchings by the Fibonacci numbers.

\begin{theorem}{\rm (\cite{Horak2022})}\label{the_extr_matc}
For $n\geq 6$, $H_{2n}$ is the unique graph with the maximum number of perfect matchings among all connected cubic bipartite graphs on $2n$ vertices.
\end{theorem}

It is worth mentioning that, for small $n\in \{3,4,5\}$, the extremal graphs with minimum algebraic connectivity and
maximum number of perfect matchings are exhibited in Figures \ref{fig_ma} and \ref{fig_mm}, respectively.
For simplicity of the statement, we omit these small cases in Theorems \ref{the_extr} and \ref{the_extr_matc}.

\begin{figure}[ht]
  \centering
  \includegraphics[scale=0.7]{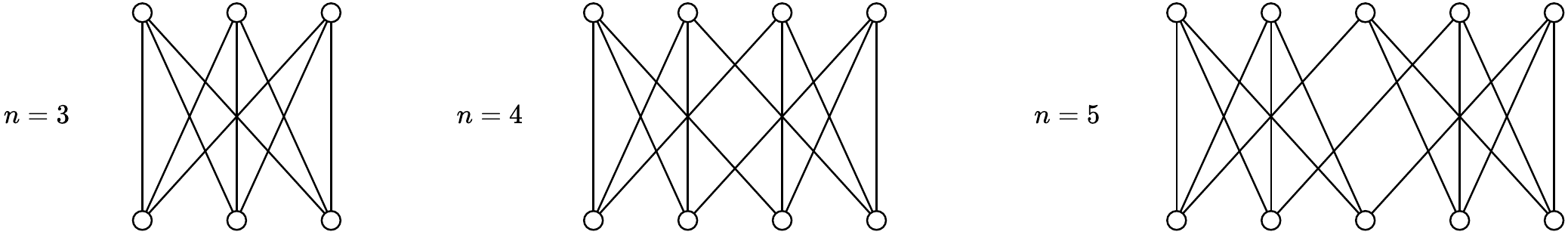}
  \caption{Cubic bipartite graphs with minimum algebraic connectivity.}
  \label{fig_ma}
  \end{figure}

  \begin{figure}[ht]
    \centering
    \includegraphics[scale=0.7]{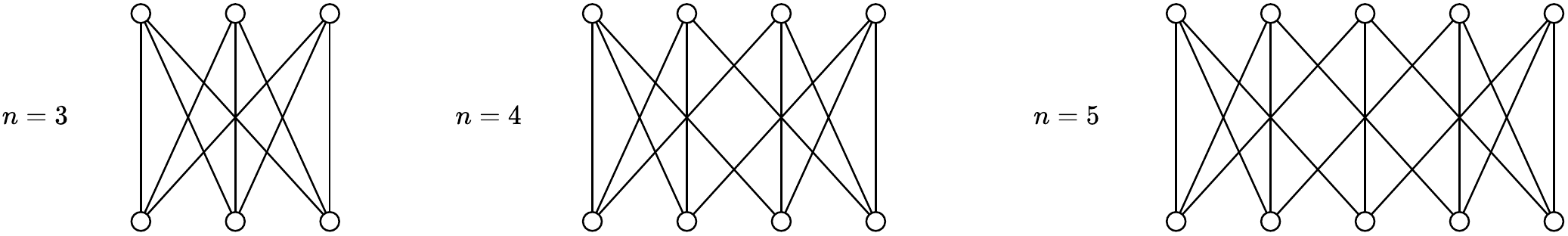}
    \caption{Cubic bipartite graphs with maximum number of perfect matchings.}
    \label{fig_mm}
    \end{figure}

Combining Theorems \ref{the_extr} and \ref{the_extr_matc},
we obtain immediately a spectral characterization for cubic bipartite graphs with maximum number of perfect matchings.

\begin{theorem}\label{the_equ}
Among all connected cubic bipartite graphs on at least $12$ vertices, the graph attains maximum number of perfect matchings
if and only if it has the minimum algebraic connectivity (spectral gap).
\end{theorem}

The rest of the paper is organized as follows. In the next section, some properties on the Fiedler vector are obtained,
and the unique connected cubic bipartite graph with minimum algebraic connectivity is determined. Moreover, we provide the proof of Theorem \ref{the_extr}.
The proof of Theorem \ref{the_min_val} is presented in Section \ref{ex_value}.
In the final section, we give equivalent results of Theorems \ref{the_min_val} and \ref{the_extr} for the spectral gap.

\section{Extremal graph with minimum algebraic connectivity}\label{ex_graph}

Let us recall some basic properties for the Laplacian eigenvalues of graphs. Let $G$ be a graph on $n$ vertices with minimum degree $\delta(G)$.
Clearly, 0 is the smallest Laplacian eigenvalue of $G$, and all ones vector $\mathtt{1}$ is the corresponding eigenvector.
Let $\mathtt{x}$ be an eigenvector of $G$ corresponding to the algebraic connectivity $a(G)$.
Such an eigenvector $\mathtt{x}$ is also called the Fiedler vector of $G$.
For any vertex $v\in V(G)$, the entry of $\mathtt{x}$ corresponding to $v$ is denoted by $\mathtt{x}_{v}$.
Note that $\mathtt{x}\bot \mathtt{1}$. According to Courant-Fischer Theorem, one can see that
\begin{eqnarray}\label{eq1}
a(G)=\min_{\mathtt{z}\bot \mathtt{1}}\frac{\mathtt{z}^{t}L(G)\mathtt{z}}{\mathtt{z}^{t}\mathtt{z}}
=\frac{\mathtt{x}^{t}L(G)\mathtt{x}}{\mathtt{x}^{t}\mathtt{x}}
=\frac{1}{\mathtt{x}^{t}\mathtt{x}}\sum_{uv\in E(G)}(\mathtt{x}_{u}-\mathtt{x}_{v})^{2}.
\end{eqnarray}
It is well-known that the algebraic connectivity is not greater than the minimum degree, that is,
$$a(G)\leq \delta(G).$$
Denote by $G^{c}$ the complement of $G$. The Laplacian spectral radius of $G^{c}$ is written as $\mu(G^{c})$.
Obviously, the maximum degree of $G^{c}$ is $\Delta(G^{c})=n-1-\delta(G)$. A famous lower bound for the Laplacian spectral radius is
$$\mu(G^{c})\geq \Delta(G^{c})+1,$$
with equality if and only if $\Delta(G^{c})=n-1$. If $G$ is connected, then $\Delta(G^{c})<n-1$, and so $\mu(G^{c})>\Delta(G^{c})+1=n-\delta(G)$.
Note also that $a(G)$ and $\mu(G^{c})$ satisfy
$$a(G)+\mu(G^{c})=n.$$
Combining the above facts, one can see that
\begin{eqnarray}\label{eq2}
a(G)<\delta(G)
\end{eqnarray}
if $G$ is connected.

Given a graph $G$, let $\{u,v,u',v'\}$ be four distinct vertices in $G$ satisfying the following condition:
$$u\sim v, u'\sim v', u\nsim u', u\nsim v', v\nsim u', v\nsim v'.$$
Clearly, the induced subgraph $G[\{u,v,u',v'\}]$ is isomorphic to $2K_{2}$. Then we say that $\{uv, u'v'\}$ is a pair of independent edges.

\begin{lemma}\label{lem_trans}
  Let $G$ be a connected graph with a pair of independent edges $\{u_{1}v_{1},u_{2}v_{2}\}$.
  Suppose that $\mathtt{x}$ is a Fiedler vector of $G$ and $d_{G}(v_{1})=d_{G}(v_{2})$.
  Let $G'$ be a connected graph obtained from $G$ by deleting edges $\{u_{1}v_{1},u_{2}v_{2}\}$
  and adding edges $\{u_{1}v_{2},u_{2}v_{1}\}$.
  If $\mathtt{x}_{u_{1}}>\mathtt{x}_{u_{2}}$ and $\mathtt{x}_{v_{1}}\leq \mathtt{x}_{v_{2}}$, then $a(G')<a(G)$.
\end{lemma}

\begin{proof}
  We may assume that $\mathtt{x}$ is a unit Fiedler vector. If $\mathtt{x}_{v_{1}}<\mathtt{x}_{v_{2}}$, then
  \begin{eqnarray*}
    \mathtt{x}^{t}L(G')\mathtt{x}-\mathtt{x}^{t}L(G)\mathtt{x}
    &=&(\mathtt{x}_{u_{1}}-\mathtt{x}_{v_{2}})^{2}+(\mathtt{x}_{u_{2}}-\mathtt{x}_{v_{1}})^{2}
    -(\mathtt{x}_{u_{1}}-\mathtt{x}_{v_{1}})^{2}-(\mathtt{x}_{u_{2}}-\mathtt{x}_{v_{2}})^{2}\\
    &=&2\mathtt{x}_{u_{1}}\mathtt{x}_{v_{1}}+2\mathtt{x}_{u_{2}}\mathtt{x}_{v_{2}}
    -2\mathtt{x}_{u_{1}}\mathtt{x}_{v_{2}}-2\mathtt{x}_{u_{2}}\mathtt{x}_{v_{1}}\\
    &=&2(\mathtt{x}_{u_{1}}-\mathtt{x}_{u_{2}})(\mathtt{x}_{v_{1}}-\mathtt{x}_{v_{2}})\\
    &<&0,
  \end{eqnarray*}
  where the last inequality holds since $\mathtt{x}_{u_{1}}>\mathtt{x}_{u_{2}}$ and $\mathtt{x}_{v_{1}}<\mathtt{x}_{v_{2}}$.
  According to (\ref{eq1}), it follows that
  $$a(G')\leq \mathtt{x}^{t}L(G')\mathtt{x}<\mathtt{x}^{t}L(G)\mathtt{x}=a(G),$$
  as required.

  In the following, we assume that $\mathtt{x}_{v_{1}}=\mathtt{x}_{v_{2}}$.
  Suppose that $d_{G}(v_{1})=d_{G}(v_{2})=k$. Then the minimum degree of $G$ is at most $k$.
  It follows from (\ref{eq2}) that $a(G)<k$. We construct a new vector $\mathtt{z}$ such that
  \begin{equation*}
    \left\{
      \begin{aligned}
        &\mathtt{z}_{v_{1}}=\mathtt{x}_{v_{1}}+\frac{\mathtt{x}_{u_{2}}-\mathtt{x}_{u_{1}}}{k-a(G)},\\
        &\mathtt{z}_{v_{2}}=\mathtt{x}_{v_{1}}+\frac{\mathtt{x}_{u_{1}}-\mathtt{x}_{u_{2}}}{k-a(G)},\\
        &\mathtt{z}_{w}=\mathtt{x}_{w}~~\text{for}~~w\in V(G)\backslash\{v_{1},v_{2}\}.
      \end{aligned}
    \right.
  \end{equation*}
  It is easy to see that
  $$\mathtt{z}^{t}\mathtt{1}=\mathtt{x}^{t}\mathtt{1}+\frac{\mathtt{x}_{u_{2}}-\mathtt{x}_{u_{1}}}{k-a(G)}
  +\frac{\mathtt{x}_{u_{1}}-\mathtt{x}_{u_{2}}}{k-a(G)}=\mathtt{x}^{t}\mathtt{1},$$
  hence $\mathtt{z}\bot\mathtt{1}$.
  Set $N_{1}=\{w:w\in N_{G}(v_{1})\backslash\{u_{1}\}\}$ and  $N_{2}=\{w:w\in N_{G}(v_{2})\backslash\{u_{2}\}\}$.
  In the graph $G'$, the neighborhood of $v_{1}$ and $v_{2}$ are $N_{1}\cup\{u_{2}\}$ and $N_{2}\cup\{u_{1}\}$, respectively.
  Then we obtain that
  \begin{eqnarray*}
    &&\mathtt{z}^{t}L(G')\mathtt{z}-\mathtt{x}^{t}L(G)\mathtt{x}\\
    &=&\sum_{w\in N_{1}}(\mathtt{z}_{v_{1}}-\mathtt{z}_{w})^{2}+\sum_{w\in N_{2}}(\mathtt{z}_{v_{2}}-\mathtt{z}_{w})^{2}
    -\sum_{w\in N_{1}}(\mathtt{x}_{v_{1}}-\mathtt{x}_{w})^{2}+\sum_{w\in N_{2}}(\mathtt{x}_{v_{2}}-\mathtt{x}_{w})^{2}\\
    &&+(\mathtt{z}_{v_{1}}-\mathtt{z}_{u_{2}})^{2}+(\mathtt{z}_{v_{2}}-\mathtt{z}_{u_{1}})^{2}
    -(\mathtt{x}_{v_{1}}-\mathtt{x}_{u_{1}})^{2}-(\mathtt{x}_{v_{2}}-\mathtt{x}_{u_{2}})^{2}\\
    &=& \sum_{w\in N_{1}}(\mathtt{z}_{v_{1}}-\mathtt{x}_{v_{1}})(\mathtt{z}_{v_{1}}+\mathtt{x}_{v_{1}}-2\mathtt{x}_{w})
    +\sum_{w\in N_{2}}(\mathtt{z}_{v_{2}}-\mathtt{x}_{v_{2}})(\mathtt{z}_{v_{2}}+\mathtt{x}_{v_{2}}-2\mathtt{x}_{w})\\
    && +(\mathtt{z}_{v_{1}}-\mathtt{x}_{v_{2}})(\mathtt{z}_{v_{1}}+\mathtt{x}_{v_{2}}-2\mathtt{x}_{u_{2}})
    +(\mathtt{z}_{v_{2}}-\mathtt{x}_{v_{1}})(\mathtt{z}_{v_{2}}+\mathtt{x}_{v_{1}}-2\mathtt{x}_{u_{1}})\\
    &=& (\mathtt{z}_{v_{1}}-\mathtt{x}_{v_{1}})\left(k(\mathtt{z}_{v_{1}}+\mathtt{x}_{v_{1}})-2\mathtt{x}_{u_{2}}-2\sum_{w\in N_{1}}\mathtt{x}_{w}\right)
    +(\mathtt{z}_{v_{2}}-\mathtt{x}_{v_{1}})\left(k(\mathtt{z}_{v_{2}}+\mathtt{x}_{v_{1}})-2\mathtt{x}_{u_{1}}-2\sum_{w\in N_{2}}\mathtt{x}_{w}\right)\\
    &=& (\mathtt{z}_{v_{1}}-\mathtt{x}_{v_{1}})\left(k(\mathtt{z}_{v_{1}}-\mathtt{z}_{v_{2}})-2\mathtt{x}_{u_{2}}-2\sum_{w\in N_{1}}\mathtt{x}_{w}
    +2\mathtt{x}_{u_{1}}+2\sum_{w\in N_{2}}\mathtt{x}_{w}\right).
  \end{eqnarray*}
  Moreover, according to $L(G)\mathtt{x}=a(G)\mathtt{x}$, we have
  $$a(G)\mathtt{x}_{v_{1}}=k\mathtt{x}_{v_{1}}-\mathtt{x}_{u_{1}}-\sum_{w\in N_{1}}\mathtt{x}_{w}~~\text{and}~~
  a(G)\mathtt{x}_{v_{2}}=k\mathtt{x}_{v_{2}}-\mathtt{x}_{u_{2}}-\sum_{w\in N_{2}}\mathtt{x}_{w}.$$
  Since $\mathtt{x}_{v_{1}}=\mathtt{x}_{v_{2}}$, it follows that
  $$\mathtt{x}_{u_{1}}+\sum_{w\in N_{1}}\mathtt{x}_{w}=\mathtt{x}_{u_{2}}+\sum_{w\in N_{2}}\mathtt{x}_{w}.$$
  Therefore, we obtain that
  \begin{eqnarray*}
    \mathtt{z}^{t}L(G')\mathtt{z}-\mathtt{x}^{t}L(G)\mathtt{x}&=&(\mathtt{z}_{v_{1}}-\mathtt{x}_{v_{1}})\left(k(\mathtt{z}_{v_{1}}-\mathtt{z}_{v_{2}})+4\mathtt{x}_{u_{1}}-4\mathtt{x}_{u_{2}}\right)\\
    &=&\frac{\mathtt{x}_{u_{2}}-\mathtt{x}_{u_{1}}}{k-a(G)}\left(\frac{2k(\mathtt{x}_{u_{2}}-\mathtt{x}_{u_{1}})}{k-a(G)}-4(\mathtt{x}_{u_{2}}-\mathtt{x}_{u_{1}})\right)\\
    &=&\frac{(4a(G)-2k)(\mathtt{x}_{u_{2}}-\mathtt{x}_{u_{1}})^{2}}{(k-a(G))^{2}}\\
    &<&\frac{2a(G)(\mathtt{x}_{u_{2}}-\mathtt{x}_{u_{1}})^{2}}{(k-a(G))^{2}},
  \end{eqnarray*}
  where the last inequality follows from the fact $a(G)<k$.
  Note that
  \begin{eqnarray*}
    \mathtt{z}^{t}\mathtt{z}=\mathtt{x}^{t}\mathtt{x}-2\mathtt{x}_{v_{1}}^{2}+\mathtt{z}_{v_{1}}^{2}+\mathtt{z}_{v_{2}}^{2}
    =1+\frac{2(\mathtt{x}_{u_{2}}-\mathtt{x}_{u_{1}})^{2}}{(k-a(G))^{2}}.
  \end{eqnarray*}
  According to (\ref{eq1}), it follows that
  \begin{eqnarray*}
    a(G')\leq \frac{\mathtt{z}^{t}L(G')\mathtt{z}}{\mathtt{z}^{t}\mathtt{z}}
    <\frac{\mathtt{x}^{t}L(G)\mathtt{x}+\frac{2a(G)(\mathtt{x}_{u_{2}}-\mathtt{x}_{u_{1}})^{2}}{(k-a(G))^{2}}}{1+\frac{2(\mathtt{x}_{u_{2}}-\mathtt{x}_{u_{1}})^{2}}{(k-a(G))^{2}}}
    =\frac{a(G)+\frac{2a(G)(\mathtt{x}_{u_{2}}-\mathtt{x}_{u_{1}})^{2}}{(k-a(G))^{2}}}{1+\frac{2(\mathtt{x}_{u_{2}}-\mathtt{x}_{u_{1}})^{2}}{(k-a(G))^{2}}}
    =a(G).
  \end{eqnarray*}
  This completes the proof.
\end{proof}

The following useful lemma on the Fiedler vector is due to Fiedler \cite{Fiedler1975}.

\begin{lemma}{\rm (\cite{Fiedler1975})}\label{lem_con}
  Let $G$ be a connected graph with a Fiedler vector $\mathtt{x}$. For any $r\geq 0$, let
  $$[V(G)]_{\geq -r}=\{v\in V(G):\mathtt{x}_{v}\geq -r\},~~[V(G)]_{\leq r}=\{v\in V(G):\mathtt{x}_{v}\leq r\}.$$
  Then the subgraphs induced by $[V(G)]_{\geq -r}$ and $[V(G)]_{\leq r}$ are connected.
\end{lemma}

Let $\mathcal{B}(2n,3)$ be the set of all connected cubic bipartite graphs on $2n$ vertices.
A graph in $\mathcal{B}(2n,3)$ is called extremal if it has the minimum algebraic connectivity.
The aim of this section is to determine the extremal graph in $\mathcal{B}(2n,3)$.
Suppose that $G$ is an extremal graph in $\mathcal{B}(2n,3)$ with bipartition $(U,V)$.
Let $\{uv,u'v'\}$ be a pair of independent edges in $G$ with $\{u,u'\}\subseteq U$ and $\{v,v'\}\subseteq V$.
We define four vertex sets $\hat{N}(u)=N_{G}(u)\backslash\{v\}$, $\hat{N}(u')=N_{G}(u')\backslash\{v'\}$, $\hat{N}(v)=N_{G}(v)\backslash\{u\}$
and $\hat{N}(v')=N_{G}(v')\backslash\{u'\}$.

The following lemmas present some properties on the Fiedler vector of the extremal graph $G$.

\begin{lemma}\label{lem_prop}
Let $G$ be an extremal graph in $\mathcal{B}(2n,3)$. Suppose that $\{uv,u'v'\}$ is a pair of independent edges in $G$ as defined above.
Let $\mathtt{x}$ be a unit Fiedler vector of $G$. If $\mathtt{x}_{u}>\mathtt{x}_{u'}$ and $\mathtt{x}_{v}\leq \mathtt{x}_{v'}$, then the following statements hold.\\
\noindent {\rm (P1)} $G$ contains no cut edge.\\
\noindent {\rm (P2)} $\{uv,u'v'\}$ is an edge cut.\\
\noindent {\rm (P3)} $N_{G}(u)\cap N_{G}(u')=\emptyset$ and $N_{G}(v)\cap N_{G}(v')=\emptyset$.\\
\noindent {\rm (P4)} $\min\{\mathtt{x}_{w}:w\in \hat{N}(u)\}>\mathtt{x}_{v'}\geq \mathtt{x}_{v}>\max\{\mathtt{x}_{w}:w\in \hat{N}(v)\}$.\\
\noindent {\rm (P5)} $\min\{\mathtt{x}_{w}:w\in \hat{N}(v')\}\geq \mathtt{x}_{u}>\mathtt{x}_{u'}\geq \max\{\mathtt{x}_{w}:w\in \hat{N}(u')\}$.
\end{lemma}

\begin{proof}
If there exists a cut edge $e$ in $G$, then a component of $G-e$ has the degree sequence $(2,3,\ldots, 3)$.
Note that such a component is also a bipartite graph. However, the sequence $(2,3,\ldots, 3)$ cannot be the degree sequence of any bipartite graph,
hence there is no cut edge in $G$, and so (P1) holds.

We claim that the graph $G-\{uv,u'v'\}+\{uv',u'v\}$ must be disconnected. Otherwise,
if $G-\{uv,u'v'\}+\{uv',u'v\}$ is connected, then Lemma \ref{lem_trans} shows that
the algebraic connectivity of $G-\{uv,u'v'\}+\{uv',u'v\}$ is less than the algebraic connectivity of $G$, a contradiction.

If $G-\{uv,u'v'\}$ is connected, then the graph $G-\{uv,u'v'\}+\{uv',u'v\}$ is clearly connected, a contradiction.
Therefore, $\{uv,u'v'\}$ is an edge cut. So (P2) holds.

According to (P1) and (P2), there are exactly two components in $G-\{uv,u'v'\}$, say $H_{1}$ and $H_{2}$.
If $H_{1}$ (or $H_{2}$) contains only one vertex of $\{u,u',v,v'\}$, then there is a cut edge in $G$, contradicting (P1).
Thus, $|V(H_{1})\cap \{u,u',v,v'\}|=|V(H_{2})\cap \{u,u',v,v'\}|=2$. If $u$ and $u'$ belong to the same component, then $v$ and $v'$ are included in the other component.
However, in this case, the graph $G-\{uv,u'v'\}+\{uv',u'v\}$ is still connected, a contradiction.
Therefore, without loss of generality, we may assume that $\{u,v'\}\subseteq V(H_{1})$ and $\{u',v\}\subseteq V(H_{2})$.
In view of the above assumption, the local structure of $G$ is exhibited in Figure \ref{fig1}.
It is easy to see that $N_{G}(u)\cap N_{G}(u')=\emptyset$ and $N_{G}(v)\cap N_{G}(v')=\emptyset$, and (P3) follows.

Choose a vertex $w\in \hat{N}(u)$. Clearly, $\{uw,u'v'\}$ is also a pair of independent edges in $G$.
Let $u'Pv$ be a path in $H_{2}$. Consider the graph $G-\{uw,u'v'\}+\{uv',u'w\}$.
In this graph, the vertices $u,u',v',w$ are connected by the path $wu'Pvuv'$.
Hence $G-\{uw,u'v'\}+\{uv',u'w\}$ is connected. If $\mathtt{x}_{w}\leq \mathtt{x}_{v'}$,
then it follows from Lemma \ref{lem_trans} that
the algebraic connectivity of $G-\{uw,u'v'\}+\{uv',u'w\}$ is less than the algebraic connectivity of $G$, a contradiction.
Therefore, we obtain that $\mathtt{x}_{w}>\mathtt{x}_{v'}$, and so $\min\{\mathtt{x}_{w}:w\in \hat{N}(u)\}>\mathtt{x}_{v'}$.
A similar argument, using Lemma \ref{lem_trans}, yields that $\mathtt{x}_{v}>\max\{\mathtt{x}_{w}:w\in \hat{N}(v)\}$,
$\min\{\mathtt{x}_{w}:w\in \hat{N}(v')\}\geq \mathtt{x}_{u}$ and $\mathtt{x}_{u'}\geq \max\{\mathtt{x}_{w}:w\in \hat{N}(u')\}$.
Now statements (P4) and (P5) hold directly.
\end{proof}

\begin{figure}[ht]
  \centering
  \includegraphics[scale=0.5]{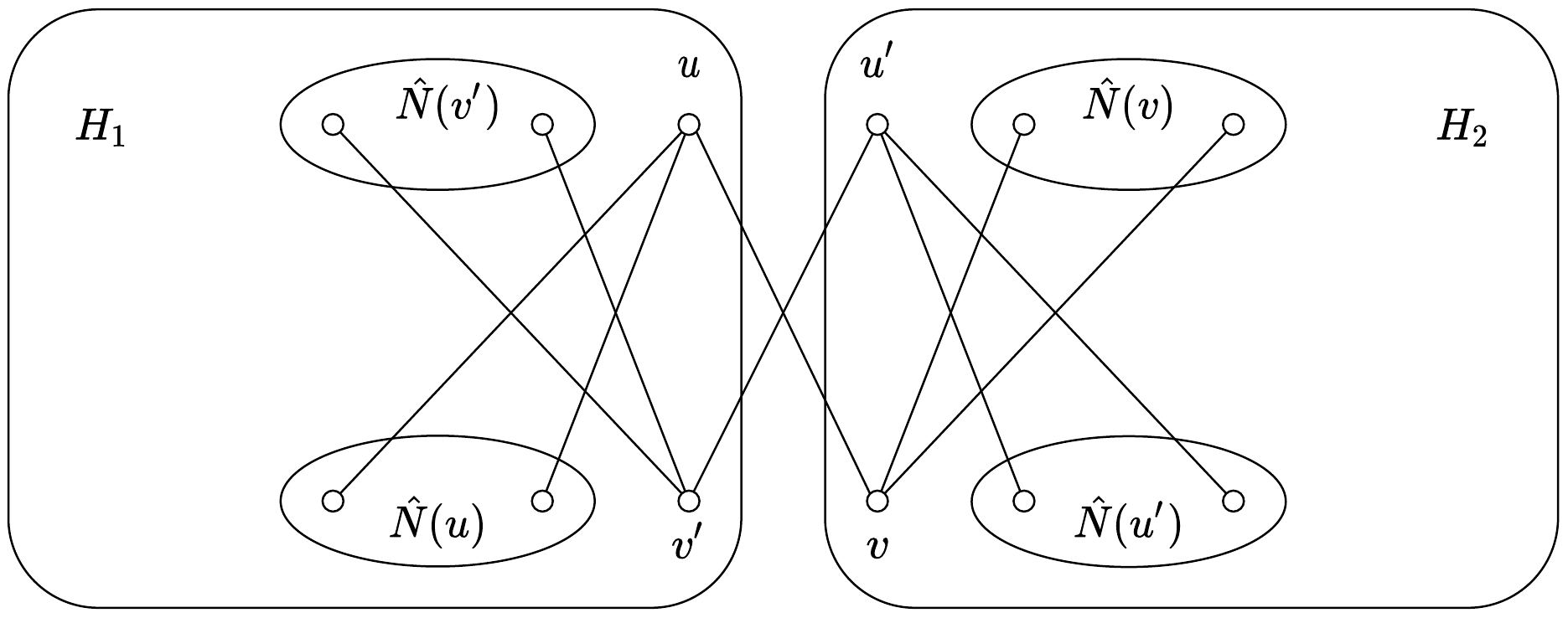}
  \caption{Local structure of the extremal graph.}
  \label{fig1}
  \end{figure}

\begin{lemma}\label{lem_adja}
  Let $G$ be an extremal graph in $\mathcal{B}(2n,3)$. Suppose that $u$ and $v$ are two nonadjacent vertices belonging to different parts of $G$.
  Let $u',v',v''$ be three vertices such that $u\sim v'$, $u\sim v''$ and $u'\sim v$. \\
  (1)  If $\mathtt{x}_{u}>\mathtt{x}_{u'}$ and $\mathtt{x}_{v}\geq \max\{\mathtt{x}_{v'},\mathtt{x}_{v''}\}$,
  then $u'$ is adjacent to both $v'$ and $v''$. \\
  (2)  If $\mathtt{x}_{u}\geq \mathtt{x}_{u'}$ and $\mathtt{x}_{v}>\max\{\mathtt{x}_{v'},\mathtt{x}_{v''}\}$,
  then $u'$ is adjacent to both $v'$ and $v''$.
\end{lemma}

\begin{proof}
  We prove (1) by contradiction. Suppose that $u'\nsim v'$. Hence $\{uv',u'v\}$ is a pair of independent edges in $G$.
  Note that $\mathtt{x}_{u}>\mathtt{x}_{u'}$ and $\mathtt{x}_{v}\geq \mathtt{x}_{v'}$.
  By Lemma \ref{lem_prop} (P4), we obtain that $\min\{\mathtt{x}_{w}:w\in N_{G}(u)\backslash\{v'\}\}>\mathtt{x}_{v}$.
  Since $v''\in N_{G}(u)\backslash\{v'\}$, it follows that $\mathtt{x}_{v''}>\mathtt{x}_{v}$,
  contradicting the fact that $\mathtt{x}_{v}\geq \max\{\mathtt{x}_{v'},\mathtt{x}_{v''}\}$.
  A similar argument, using Lemma \ref{lem_prop} (P5), yields a proof of (2).
\end{proof}

\begin{lemma}\label{lem_max1}
  Let $G\in \mathcal{B}(2n,3)$ be an extremal graph with the bipartition $(U,V)$.
  If $u^{*}$ is a vertex in $U$ such that $\mathtt{x}_{u^{*}}=\max\{\mathtt{x}_{u}:u\in U\}$,
  then $\min\{\mathtt{x}_{v}:v\in N_{G}(u^{*})\}\geq \max\{\mathtt{x}_{w}:w\in V\backslash N_{G}(u^{*})\}$.
\end{lemma}
\begin{proof}
  Suppose that $N_{G}(u^{*})=\{v_{1},v_{2},v_{3}\}$ and $\mathtt{x}_{v_{1}}\geq \mathtt{x}_{v_{2}}\geq \mathtt{x}_{v_{3}}$.
  Let $v_{4}$ be a vertex in $V\backslash N_{G}(u^{*})$ with $\mathtt{x}_{v_{4}}=\max\{\mathtt{x}_{w}:w\in V\backslash N_{G}(u^{*})\}$.
  We only need to show that $\mathtt{x}_{v_{3}}\geq\mathtt{x}_{v_{4}}$. We prove it by contradiction. Suppose that $\mathtt{x}_{v_{3}}<\mathtt{x}_{v_{4}}$.
  Since the degree of $v_{3}$ is 3, then there is a vertex $u'\in U$ such that $u'\sim v_{4}$ and $u'\nsim v_{3}$.
  Clearly, $\{u^{*}v_{3},u'v_{4}\}$ is a pair of independent edges in $G$.
  Since $\mathtt{x}_{v_{3}}<\mathtt{x}_{v_{4}}$ and $\mathtt{x}_{u^{*}}\geq \mathtt{x}_{u'}$,
  by Lemma \ref{lem_prop}(P4), we have $\min\{\mathtt{x}_{w}:w\in N_{G}(v_{4})\backslash\{u'\}\}>\mathtt{x}_{u^{*}}$.
  which contradicts the maximality of $u^{*}$. Hence the result follows.
\end{proof}

Similar to the proof of Lemma \ref{lem_max1}, we can obtain the following result.

\begin{lemma}\label{lem_max2}
  Let $G\in \mathcal{B}(2n,3)$ be an extremal graph with the bipartition $(U,V)$.
  If $v^{*}$ is a vertex in $V$ such that $\mathtt{x}_{v^{*}}=\max\{\mathtt{x}_{v}:v\in V\}$,
  then $\min\{\mathtt{x}_{u}:u\in N_{G}(v^{*})\}\geq \max\{\mathtt{x}_{w}:w\in V\backslash N_{G}(v^{*})\}$.
\end{lemma}

In the following, we determine the structure of the extremal graph $G\in \mathcal{B}(2n,3)$.
For $n\leq 5$, it is easy to determine the extremal graphs by a directly computation.
As mentioned in Section 1, the extremal graphs for small $n\in \{3,4,5\}$ are presented in Figure \ref{fig_ma}.
Next we only need to consider the case $n\geq 6$, that is, there are at least 12 vertices in $G$.
Suppose that the bipartition of $G$ is $(U,V)$,
where $U=\{u_{1},u_{2},\ldots, u_{n}\}$ and $V=\{v_{1},v_{2},\ldots, v_{n}\}$.

Let $\mathtt{x}$ be a unit Fiedler vector of $G$. According to Lemmas \ref{lem_max1} and \ref{lem_max2},
after a relabeling of the vertices of $G$, we may assume that the vertices satisfying:
\begin{itemize}
  \item[(1)] $\mathtt{x}_{u_{1}}\geq \mathtt{x}_{u_{2}}\geq \cdots \geq \mathtt{x}_{u_{n}}$,
  \item[(2)] $\mathtt{x}_{v_{1}}\geq \mathtt{x}_{v_{2}}\geq \cdots \geq \mathtt{x}_{v_{n}}$,
  \item[(3)] $N_{G}(u_{1})=\{v_{1},v_{2},v_{3}\}$,
  \item[(4)] $N_{G}(v_{1})=\{u_{1},u_{2},u_{3}\}$.
\end{itemize}

Let $S$ be a vertex subset of the extremal graph $G$.
We denote by $\mathtt{x}[S]$ the sum of the entries of the Fiedler vector $\mathtt{x}$ corresponding to the vertices in $S$, that is,
$$\mathtt{x}[S]=\sum_{w\in S}\mathtt{x}_{w}.$$
The following basic property on the Fiedler vector $\mathtt{x}$ is useful.

\begin{lemma}\label{lem_neig}
  Let $u$ and $v$ be any two vertices in $G$. Then\\
  \noindent {\rm (i)} $\mathtt{x}_{u}=\mathtt{x}_{v}$ if and only if $\mathtt{x}[N_{G}(u)]=\mathtt{x}[N_{G}(v)]$,\\
  \noindent {\rm (ii)} $\mathtt{x}_{u}>\mathtt{x}_{v}$ if and only if $\mathtt{x}[N_{G}(u)]>\mathtt{x}[N_{G}(v)]$.
\end{lemma}
\begin{proof}
  Note that $a(G)\mathtt{x}=L(G)\mathtt{x}.$ It follows that
  $$a(G)\mathtt{x}_{u}=3\mathtt{x}_{u}-\sum_{w\in N_{G}(u)}\mathtt{x}_{w},$$
  that is,
  $$(3-a(G))\mathtt{x}_{u}=\sum_{w\in N_{G}(v)}\mathtt{x}_{w}=\mathtt{x}[N_{G}(u)].$$
  Similarly, we can obtain
  $$(3-a(G))\mathtt{x}_{v}=\mathtt{x}[N_{G}(v)].$$
  According to (\ref{eq2}), we have $3-a(G)>0$, and hence the result follows immediately.
\end{proof}

In order to determine the structure of the extremal graph $G$, we need to establish the adjacency rule of the first six vertices $\{u_{1},u_{2},u_{3},v_{1},v_{2},v_{3}\}$.

\begin{lemma}\label{lem_pos}
  $\mathtt{x}_{u_{1}}>0$ and $\mathtt{x}_{v_{1}}>0$.
\end{lemma}
\begin{proof}
  If $\mathtt{x}_{u_{1}}\leq 0$, then
  $$(3-a(G))\mathtt{x}_{v_{1}}=\mathtt{x}[N_{G}(v_{1})]=\mathtt{x}_{u_{1}}+\mathtt{x}_{u_{2}}+\mathtt{x}_{u_{3}}\leq 3\mathtt{x}_{u_{1}}\leq 0.$$
  Since $a(G)<3$, it follows that $\mathtt{x}_{v_{1}}\leq 0$.
  This implies that the vector $\mathtt{x}$ is non-positive, which contradicts the definition of the Fiedler vector.
  It follows that $\mathtt{x}_{u_{1}}>0$. Similarly, $\mathtt{x}_{v_{1}}>0$ since $(3-a(G))\mathtt{x}_{u_{1}}\leq 3\mathtt{x}_{v_{1}}$.
\end{proof}

\begin{lemma}\label{lem_neq}
  $\mathtt{x}_{u_{1}}\neq \mathtt{x}_{u_{3}}$ and $\mathtt{x}_{v_{1}}\neq \mathtt{x}_{v_{3}}$.
\end{lemma}

\begin{proof}
Since these two inequalities can be proved by the same approach, we only present the proof of the inequality $\mathtt{x}_{u_{1}}\neq \mathtt{x}_{u_{3}}$.

We establish the inequality $\mathtt{x}_{u_{1}}\neq \mathtt{x}_{u_{3}}$ by contradiction.
Suppose that $\mathtt{x}_{u_{1}}=\mathtt{x}_{u_{3}}$. Hence $\mathtt{x}_{u_{1}}=\mathtt{x}_{u_{2}}=\mathtt{x}_{u_{3}}$.
Note that $(3-a(G))\mathtt{x}_{v_{1}}=\mathtt{x}[N_{G}(v_{1})]=3\mathtt{x}_{u_{1}}$. Since $\mathtt{x}_{v_{1}}>0$ and $a(G)<3$,
 we have $\mathtt{x}_{v_{1}}>\mathtt{x}_{u_{1}}>0$.
 Moreover, one can see that
 $$3\mathtt{x}_{v_{3}}\leq\mathtt{x}_{v_{1}}+\mathtt{x}_{v_{2}}+\mathtt{x}_{v_{3}}=(3-a(G))\mathtt{x}_{u_{1}}<3\mathtt{x}_{u_{1}},$$
 and so $\mathtt{x}_{v_{3}}<\mathtt{x}_{u_{1}}$. Hence $\mathtt{x}_{v_{1}}>\mathtt{x}_{v_{3}}$.

 Consider the vertex $u_{4}$. Clearly, $u_{4}\nsim v_{1}$. Thus, we obtain that
 $$\mathtt{x}[N_{G}(u_{4})]\leq \mathtt{x}_{v_{2}}+2\mathtt{x}_{v_{3}}<\mathtt{x}_{v_{1}}+\mathtt{x}_{v_{2}}+\mathtt{x}_{v_{3}}=\mathtt{x}[N_{G}(u_{1})].$$
 By Lemma \ref{lem_neig}, it follows that $\mathtt{x}_{u_{4}}<\mathtt{x}_{u_{1}}$.

 In summary, we obtain that
 $$\mathtt{x}_{u_{1}}=\mathtt{x}_{u_{2}}=\mathtt{x}_{u_{3}}>\mathtt{x}_{u_{4}}~~~\text{and}~~~\mathtt{x}_{v_{1}}>\mathtt{x}_{v_{3}}.$$
 Next we will divide the proof into the following two cases.

\medskip

\noindent{\bf Case 1.} $\mathtt{x}_{v_{2}}\neq \mathtt{x}_{v_{3}}$,
 i.e., $\mathtt{x}_{v_{2}}>\mathtt{x}_{v_{3}}$.

 \medskip

 If $u_{2}\nsim v_{2}$, then
 $$\mathtt{x}[N_{G}(u_{2})]\leq \mathtt{x}_{v_{1}}+2\mathtt{x}_{v_{3}}
 <\mathtt{x}_{v_{1}}+\mathtt{x}_{v_{2}}+\mathtt{x}_{v_{3}}=\mathtt{x}[N_{G}(u_{1})].$$
 This implies that $\mathtt{x}[N_{G}(u_{2})]<\mathtt{x}[N_{G}(u_{1})]$, and it follows from Lemma \ref{lem_neig} that $\mathtt{x}_{u_{1}}<\mathtt{x}_{u_{2}}$,
 contradicting the assumption that $\mathtt{x}_{u_{1}}=\mathtt{x}_{u_{2}}$. Thus, $u_{2}\sim v_{2}$.
 If $u_{3}\nsim v_{2}$, then
 $$\mathtt{x}[N_{G}(u_{3})]\leq \mathtt{x}_{v_{1}}+2\mathtt{x}_{v_{3}}
 <\mathtt{x}_{v_{1}}+\mathtt{x}_{v_{2}}+\mathtt{x}_{v_{3}}=\mathtt{x}[N_{G}(u_{1})],$$
 contradicting the assumption that $\mathtt{x}_{u_{1}}=\mathtt{x}_{u_{3}}$. Hence $u_{3}\sim v_{2}$.
 It follows that $N_{G}(v_{2})=\{u_{1},u_{2},u_{3}\}$, and so $\mathtt{x}_{v_{1}}=\mathtt{x}_{v_{2}}$.

Since $\mathtt{x}_{v_{2}}>\mathtt{x}_{v_{3}}$, by Lemma \ref{lem_neig}, we have
$\mathtt{x}[N_{G}(v_{2})]>\mathtt{x}[N_{G}(v_{3})].$
Note that $$\mathtt{x}[N_{G}(v_{2})]=\mathtt{x}_{u_{1}}+\mathtt{x}_{u_{2}}+\mathtt{x}_{u_{3}}$$ and
$$\mathtt{x}[N_{G}(v_{3})]=\mathtt{x}_{u_{1}}+\sum_{w\in N_{G}(v_{3})\backslash\{u_{1}\}}\mathtt{x}_{w}.$$
It follows that $1\leq |N_{G}(v_{3})\cap N_{G}(v_{2})|\leq 2$.

\medskip

\noindent{\bf Subcase 1.1.} $|N_{G}(v_{3})\cap N_{G}(v_{2})|=1$.

 \medskip

 In this case, $v_{3}\nsim u_{2}$ and $v_{3}\nsim u_{3}$.
 Suppose that $N_{G}(v_{3})=\{u_{1},u_{i},u_{j}\}$, where $i>j\geq 4$. Note that $N_{G}(u_{1})=\{v_{1},v_{2},v_{3}\}$.
 Let $N_{G}(u_{2})=\{v_{1},v_{2},v_{s}\}$ with $s\geq 4$.
 Since $\mathtt{x}_{u_{1}}=\mathtt{x}_{u_{2}}$, we have $\mathtt{x}[N_{G}(u_{1})]=\mathtt{x}[N_{G}(u_{2})]$,
 hence $\mathtt{x}_{v_{s}}=\mathtt{x}_{v_{3}}$.

 Assume that $v_{s}\nsim u_{i}$. Hence $\{u_{2}v_{s},u_{i}v_{3}\}$ is a pair of independent edges. Since $\mathtt{x}_{u_{2}}>\mathtt{x}_{u_{i}}$
 and $\mathtt{x}_{v_{s}}=\mathtt{x}_{v_{3}}$, it follows from Lemma \ref{lem_prop} that
 $\min\{\mathtt{x}_{u_{1}},\mathtt{x}_{u_{j}}\}\geq \mathtt{x}_{u_{2}}>\mathtt{x}_{u_{i}}$, contradicting the fact $\mathtt{x}_{u_{2}}>\mathtt{x}_{u_{j}}$.
 Therefore, we have $v_{s}\sim u_{i}$. A similar argument shows that $v_{s}\sim u_{j}$.
 Hence $N_{G}(v_{s})=\{u_{2},u_{i},u_{j}\}$.

 Let $N_{G}(u_{3})=\{v_{1},v_{2},v_{t}\}$ with $t\geq 4$ and $t\neq s$.
 Note that $\mathtt{x}_{u_{3}}=\mathtt{x}_{u_{1}}$. It follows from Lemma \ref{lem_neig} that $\mathtt{x}[N_{G}(u_{3})]=\mathtt{x}[N_{G}(u_{1})]$,
 which implies that $\mathtt{x}_{v_{t}}=\mathtt{x}_{v_{3}}$.
 Since there are at least 12 vertices in $G$, the vertex $v_{t}$ cannot be adjacent to both $u_{i}$ and $u_{j}$.
  We may assume that $v_{t}\nsim u_{i}$. Clearly, $\{u_{3}v_{t},u_{i}v_{3}\}$ is a pair of independent edges.
  Since $\mathtt{x}_{u_{3}}>\mathtt{x}_{u_{i}}$ and $\mathtt{x}_{v_{t}}=\mathtt{x}_{v_{3}}$, by Lemma \ref{lem_prop},
  we obtain that $\min\{\mathtt{x}_{u_{1}},\mathtt{x}_{u_{j}}\}\geq \mathtt{x}_{u_{3}}>\mathtt{x}_{u_{i}}$.
  But this contradicts the fact $\mathtt{x}_{u_{j}}< \mathtt{x}_{u_{3}}$.

  \medskip

\noindent{\bf Subcase 1.2.} $|N_{G}(v_{3})\cap N_{G}(v_{2})|=2$.

  \medskip

  Without loss of generality, we may assume that $v_{3}\sim u_{2}$ and $v_{3}\nsim u_{3}$. Then we have $N_{G}(u_{1})=N_{G}(u_{2})=\{v_{1},v_{2},v_{3}\}$.
  Suppose that $N_{G}(u_3)=\{v_{1},v_{2},v_{t}\}$, where $t\geq 4$. Since $\mathtt{x}_{u_{1}}=\mathtt{x}_{u_{3}}$,
  it follows from Lemma \ref{lem_neig} that $\mathtt{x}[N_{G}(u_{3})]=\mathtt{x}[N_{G}(u_{1})]$, hence $\mathtt{x}_{v_{3}}=\mathtt{x}_{v_{t}}$.

  Assume that $N_{G}(v_{3})=\{u_{1},u_{2},u_{k}\}$ with $k\geq 4$. Clearly, $v_{t}\nsim u_{1}$ and $v_{t}\nsim u_{2}$.
  If $u_{k}\sim v_{t}$, then
  $$\mathtt{x}[N_{G}(v_{t})]<2\mathtt{x}_{u_{3}}+\mathtt{x}_{u_{k}}=\mathtt{x}_{u_{1}}+\mathtt{x}_{u_{2}}+\mathtt{x}_{u_{k}}
  =\mathtt{x}[N_{G}(v_{3})].$$
By Lemma \ref{lem_neig}, it follows that $\mathtt{x}_{v_{t}}<\mathtt{x}_{v_{3}}$, a contradiction.
Thus, we have $u_{k}\nsim v_{t}$. One can see that $\{v_{3}u_{k},v_{t}u_{3}\}$ is a pair of independent edges.
Note that $\mathtt{x}_{u_{3}}>\mathtt{x}_{u_{k}}$ and $\mathtt{x}_{v_{3}}=\mathtt{x}_{v_{t}}$.
According to Lemma \ref{lem_prop}, it follows that
$$\min\{u_{1},u_{2}\}\geq \mathtt{x}_{u_{3}}>\mathtt{x}_{u_{k}}
\geq \max\{\mathtt{x}_{w}:w\in N_{G}(v_{t})\backslash\{u_{3}\}\}.$$
Thus, $$\mathtt{x}_{u_{2}}+\mathtt{x}_{u_{k}}>\sum_{w\in N_{G}(v_{t})\backslash\{u_{3}\}}\mathtt{x}_{w}.$$
Since $\mathtt{x}_{u_{1}}=\mathtt{x}_{u_{3}}$, we have
$$\mathtt{x}[N_{G}(v_{3})]=\mathtt{x}_{u_{1}}+\mathtt{x}_{u_{2}}+\mathtt{x}_{u_{k}}
>\mathtt{x}_{u_{3}}+\sum_{w\in N_{G}(v_{t})\backslash\{u_{3}\}}\mathtt{x}_{w}=\mathtt{x}[N_{G}(v_{t})].$$
Hence it follows from Lemma \ref{lem_neig} that $\mathtt{x}_{v_{3}}>\mathtt{x}_{v_{t}}$, a contradiction.

\medskip

\noindent{\bf Case 2.} $\mathtt{x}_{v_{2}}=\mathtt{x}_{v_{3}}$, i.e., $\mathtt{x}_{v_{1}}>\mathtt{x}_{v_{2}}=\mathtt{x}_{v_{3}}$.

\medskip

Consider the neighbours of $v_{2}$ and $v_{3}$. Note that $v_{2}\sim u_{1}$ and $N_{G}(v_{1})=\{u_{1},u_{2},u_{3}\}$.
This shows that $u_{1}\in N_{G}(v_{1})\cap N_{G}(v_{2})$,
and so $|N_{G}(v_{1})\cap N_{G}(v_{2})|\geq 1$. Moreover, since $\mathtt{x}_{v_{1}}>\mathtt{x}_{v_{2}}$, by Lemma \ref{lem_neig},
we have $\mathtt{x}[N_{G}(v_{1})]>\mathtt{x}[N_{G}(v_{2})]$. Therefore, the vertex $v_{2}$ cannot be adjacent to both $u_{2}$ and $u_{3}$.
It follows that $1\leq |N_{G}(v_{1})\cap N_{G}(v_{2})|\leq 2$.
Similarly, we can also obtain that $1\leq |N_{G}(v_{1})\cap N_{G}(v_{3})|\leq 2$.

\medskip

\noindent{\bf Subcase 2.1.} Either $|N_{G}(v_{1})\cap N_{G}(v_{2})|=2$ or $|N_{G}(v_{1})\cap N_{G}(v_{3})|=2$.

\medskip

Suppose first that $|N_{G}(v_{1})\cap N_{G}(v_{2})|=2$. Recall that $\mathtt{x}_{u_{1}}=\mathtt{x}_{u_{2}}=\mathtt{x}_{u_{3}}$.
Without loss of generality, we may assume that $N_{G}(v_{2})=\{u_{1},u_{2},u_{a}\}$ where $a\geq 4$.

We claim that $N_{G}(u_{3})\backslash \{v_{1}\}\subseteq N_{G}(u_{a})$. If not, suppose that there is a vertex $v_{i}\in N_{G}(u_{3})\backslash \{v_{1}\}$, where $i\geq 3$,
such that $v_{i}\nsim u_{a}$. Thus $\{u_{3}v_{i},u_{a}v_{2}\}$ is a pair of independent edges in $G$.
Note that $\mathtt{x}_{u_{3}}>\mathtt{x}_{u_{a}}$ and $\mathtt{x}_{v_{2}}\geq \mathtt{x}_{v_{i}}$.
According to Lemma \ref{lem_prop},
we obtain that $\min\{\mathtt{x}_{w}:w\in N_{G}(u_{3})\backslash \{v_{i}\}\}>\mathtt{x}_{v_{2}}$.
But, in this case, $v_{1}$ is the only one vertex with $\mathtt{x}_{v_{1}}>\mathtt{x}_{v_{2}}$, then we obtain a contradiction.

Suppose that $u_{3}\sim v_{3}$. We may assume that $N_{G}(u_{3})=\{v_{1},v_{3},v_{i}\}$ with $i\geq 4$.
If $u_{2}\sim v_{i}$, then the subgraph of $G$ induced by vertices $\{u_{1},u_{2},u_{3},u_{a},v_{1},v_{2},v_{3},v_{i}\}$ is a cubic bipartite graph,
this contradicts that $G$ contains at least 12 vertices. Hence $u_{2}\nsim v_{i}$.
Let $v_{j}$ be the neighbour of $u_{2}$ other than $v_{1}$ and $v_{2}$. Clearly, $j\geq 3$ and $j\neq i$.
Since $N_{G}(u_{3})\backslash \{v_{1}\}\subseteq N_{G}(u_{a})$, we have $N_{G}(u_{a})=\{v_{2},v_{3},v_{i}\}$, and so $u_{a}\nsim v_{j}$.
It follows that $\{u_{2}v_{j},u_{a}v_{3}\}$ is a pair of independent edges in $G$.
Since $\mathtt{x}_{u_{2}}>\mathtt{x}_{u_{a}}$ and $\mathtt{x}_{v_{3}}\geq \mathtt{x}_{v_{j}}$,
by Lemma \ref{lem_prop}, we obtain that $N_{G}(u_{2})\cap N_{G}(u_{a})=\emptyset$.
However, the vertex $v_{2}$ is adjacent to both $u_{2}$ and $u_{a}$, a contradiction.

Suppose now that $u_{3}\nsim v_{3}$. Let $N_{G}(u_{3})=\{v_{1},v_{i},v_{j}\}$ with $i\geq 3$ and $j\geq 3$.
Since $N_{G}(u_{3})\backslash \{v_{1}\}\subseteq N_{G}(u_{a})$, it follows that $N_{G}(u_{a})=\{v_{2},v_{i},v_{j}\}$.
Let $N^{*}(v_{3})=\{w\in N_{G}(v_{3}):\mathtt{x}_{w}<\mathtt{x}_{u_{3}}\}$.
Moreover, since $\mathtt{x}_{v_{3}}<\mathtt{x}_{v_{1}}$, by Lemma \ref{lem_neig}, we have $\mathtt{x}[N_{G}(v_{1})]>\mathtt{x}[N_{G}(v_{3})]$.
This implies that $|N^{*}(v_{3})|\geq 1$. Suppose that $u_{b}\in N^{*}(v_{3})$.
If $u_{b}\nsim v_{i}$, then $\{u_{3}v_{i},u_{b}v_{3}\}$ is a pair of independent edges in $G$.
Clearly, $\mathtt{x}_{u_{b}}<\mathtt{x}_{u_{3}}$ and $\mathtt{x}_{v_{3}}\geq \mathtt{x}_{v_{j}}$.
According to Lemma \ref{lem_prop}, we have $\min\{\mathtt{x}_{w}:w\in N_{G}(v_{3})\backslash\{v_{i}\}\}>\mathtt{x}_{v_{3}}$,
this leads to $\mathtt{x}_{v_{j}}>\mathtt{x}_{v_{3}}$, a contradiction. Therefore, we obtain that $u_{b}\sim v_{i}$.
Similarly, one can see that $u_{b}$ is also adjacent to $v_{j}$.
It follows that $v_{i}$ and $v_{j}$ are adjacent to all vertices in $N^{*}(v_{3})$.
On the other hand, since the degree of any vertex in $G$ is 3, we can see that $|N^{*}(v_{3})|=1$,
that is, $u_{b}$ is the only one vertex in $N^{*}(v_{3})$.
This leads to that $v_{3}$ must be adjacent to $u_{2}$.
But, in this case, the subgraph of $G$ induced by vertices $\{u_{1},u_{2},u_{3},u_{a},u_{b},v_{1},v_{2},v_{3},v_{i},v_{j}\}$ is a cubic bipartite graph,
which contradicting the fact that $G$ has at least 12 vertices.

In summary, we always obtain a contradiction if $|N_{G}(v_{1})\cap N_{G}(v_{2})|=2$.
Indeed, by a similar argument, one can also obtain a contradiction when $|N_{G}(v_{1})\cap N_{G}(v_{3})|=2$.

\medskip

\noindent{\bf Subcase 2.2.} $|N_{G}(v_{1})\cap N_{G}(v_{2})|=|N_{G}(v_{1})\cap N_{G}(v_{3})|=1$.

\medskip

Suppose that $N_{G}(v_{2})=\{u_{1},u_{a},u_{b}\}$ and $N_{G}(u_{2})=\{v_{1},v_{i},v_{j}\}$,
where $a,b,i,j$ are four integers greater than 3.
Since $\mathtt{x}_{u_{2}}>\max\{\mathtt{x}_{u_{a}},\mathtt{x}_{u_{b}}\}$
and $\mathtt{x}_{v_{2}}\geq \max\{\mathtt{x}_{v_{i}},\mathtt{x}_{v_{j}}\}$,
by Lemma \ref{lem_adja}, we obtain that $u_{a}\sim v_{i}$, $u_{a}\sim v_{j}$, $u_{b}\sim v_{i}$ and $u_{b}\sim v_{j}$.
Let $v_{k}$ be a neighbour of $u_{3}$ other than $v_{1}$. Clearly, $k\neq i$, $k\neq j$ and $k\geq 4$.
Note that $\mathtt{x}_{u_{3}}>\max\{\mathtt{x}_{u_{a}},\mathtt{x}_{u_{b}}\}$
and $\mathtt{x}_{v_{2}}\geq \mathtt{x}_{v_{k}}$.
It follows from Lemma \ref{lem_adja} that $v_{k}$ is adjacent to both $u_{a}$ and $u_{b}$.
This leads to that $\{v_{2},v_{i},v_{j},v_{k}\}\subseteq N_{G}(u_{a})$, contradicting the fact that $|N_{G}(u_{a})|=3$.
\end{proof}

\begin{lemma}\label{lem_u2v2}
  $u_{2}\sim v_{2}$.
\end{lemma}

\begin{proof} Suppose to the contrary that $u_{2}\nsim v_{2}$. To obtain a contradiction, we first prove the following claim.

  \medskip

\noindent{\bf Claim 1.} $\mathtt{x}_{u_{1}}>\mathtt{x}_{u_{2}}$ and $\mathtt{x}_{v_{1}}>\mathtt{x}_{v_{2}}$.

\begin{proof}[\bf Proof of Claim 1.]
  We first show that $\mathtt{x}_{v_{1}}>\mathtt{x}_{v_{2}}$.
  Suppose that $\mathtt{x}_{v_{1}}\leq \mathtt{x}_{v_{2}}$. Since $\mathtt{x}_{v_{1}}\geq \mathtt{x}_{v_{2}}$,
  we have $\mathtt{x}_{v_{1}}=\mathtt{x}_{v_{2}}$. According to Lemma \ref{lem_neq}, it follows that
  $\mathtt{x}_{v_{1}}=\mathtt{x}_{v_{2}}>\mathtt{x}_{v_{3}}$ and $\mathtt{x}_{u_{1}}>\mathtt{x}_{u_{3}}$.

  If $v_{2}\nsim u_{3}$, then we assume that $N_{G}(v_{2})=\{u_{1},u_{i},u_{j}\}$ where $i\geq 4$ and $j\geq 4$.
  Let $N_{G}(u_{2})=\{v_{1},v_{s},v_{t}\}$ with $i\geq 3$ and $t\geq 3$.
  Note that $\mathtt{x}_{v_{2}}>\max\{\mathtt{x}_{v_{s}},\mathtt{x}_{v_{t}}\}$
  and $\mathtt{x}_{u_{2}}\geq \max\{\mathtt{x}_{u_{i}},\mathtt{x}_{u_{j}}\}$.
  By Lemma \ref{lem_adja}, we obtain that $v_{s}\sim u_{i}$, $v_{s}\sim u_{j}$, $v_{t}\sim u_{i}$ and $v_{t}\sim u_{j}$.
  One can see that $N_{G}(u_{3})\backslash\{v_{1},v_{2},v_{s},v_{t}\}\neq \emptyset$.
  Choose a vertex $v_{k}\in N_{G}(u_{3})\backslash\{v_{1},v_{2},v_{s},v_{t}\}$.
  Clearly, $\mathtt{x}_{v_{2}}>\mathtt{x}_{v_{k}}$ and $\mathtt{x}_{u_{3}}\geq \max\{\mathtt{x}_{u_{i}},\mathtt{x}_{u_{j}}\}$.
  It follows from Lemma \ref{lem_adja} that $v_{k}\sim u_{i}$ and $v_{k}\sim u_{j}$.
  Then we obtain that $\{v_{2},v_{s},v_{t},v_{k}\}\subseteq N_{G}(u_{i})$, which implies that the degree of $u_{i}$ is at least 4, a contradiction.

  Assume that $v_{2}\sim u_{3}$.
  One can see that $N_{G}(u_{2})\backslash N_{G}(u_{3})\neq\emptyset$.
  Let $v_{k}$ be a vertex in $N_{G}(u_{2})\backslash N_{G}(u_{3})$ with $k\geq 3$.
  Clearly, $\{u_{2}v_{k},u_{3}v_{2}\}$ is a pair of independent edges in $G$.
  Note that $\mathtt{x}_{u_{2}}\geq \mathtt{x}_{u_{3}}$ and $\mathtt{x}_{v_{2}}>\mathtt{x}_{v_{k}}$.
  According to Lemma \ref{lem_prop}, we obtain that $N_{G}(u_{2})\cap N_{G}(u_{3})=\emptyset$,
  which contradicts the fact that $v_{1}$ is adjacent to both $u_{2}$ and $u_{3}$.

  Hence $\mathtt{x}_{v_{1}}>\mathtt{x}_{v_{2}}$. A similar argument shows that $\mathtt{x}_{u_{1}}>\mathtt{x}_{u_{2}}$.
  This completes the proof of Claim 1.
  \end{proof}

  It follows from Claim 1 that $\mathtt{x}_{u_{1}}>\mathtt{x}_{u_{2}}\geq \mathtt{x}_{u_{3}}$
  and $\mathtt{x}_{v_{1}}>\mathtt{x}_{v_{2}}\geq \mathtt{x}_{v_{3}}$.
  Recall that $u_{2}\nsim v_{2}$. According to the adjacency relation of $v_{2}$ and $u_{3}$,
  we divide the proof into the following two cases.

  \medskip

\noindent{\bf Case 1.} $v_{2}\sim u_{3}$.

  \medskip

  Suppose that $N_{G}(v_{2})=\{u_{1},u_{3},u_{i}\}$ where $i\geq 4$.
  Since $\mathtt{x}_{v_{1}}>\mathtt{x}_{v_{2}}$, it follows from Lemma \ref{lem_neig} that
  $\mathtt{x}[N_{G}(v_{1})]>\mathtt{x}[N_{G}(v_{2})]$, hence $\mathtt{x}_{u_{2}}>\mathtt{x}_{u_{i}}$.
  If $u_{3}\sim v_{3}$, then $N_{G}(u_{3})=\{v_{1},v_{2},v_{3}\}=N_{G}(u_{1})$.
  By Lemma \ref{lem_neig}, we have $\mathtt{x}_{u_{3}}=\mathtt{x}_{u_{1}}$, a contradiction.
  Hence $u_{3}\nsim v_{3}$.

  \medskip

\noindent{\bf Subcase 1.1.} $u_{2}\sim v_{3}$.

   \medskip

   Thus we may assume that $N_{G}(u_{2})=\{v_{1},v_{3},v_{k}\}$ with $k\geq 4$.
   Since $\mathtt{x}_{u_{2}}>\mathtt{x}_{u_{i}}$ and $\mathtt{x}_{v_{2}}\geq \mathtt{x}_{v_{3}}\geq \mathtt{x}_{v_{k}}$,
   Lemma \ref{lem_adja} shows that $u_{i}$ is adjacent to both $v_{3}$ and $v_{k}$.
   If $u_{3}\sim v_{k}$, then the subgraph induced by $\{u_{1},u_{2},u_{3},u_{i},v_{1},v_{2},v_{3},v_{k}\}$ is a cubic bipartite graph,
   which contradicts the fact that there are at least 12 vertices in $G$. Hence $u_{3}\nsim v_{k}$.
   One can see that $\{u_{2}v_{k},u_{3}v_{2}\}$ is a pair of independent edges in $G$.
   Since $\mathtt{x}_{u_{1}}>\mathtt{x}_{u_{2}}$, by Lemma \ref{lem_neig}, we obtain that $\mathtt{x}_{v_{2}}>\mathtt{x}_{v_{k}}$.
   Note also that $\mathtt{x}_{u_2}\geq \mathtt{x}_{u_{3}}$.
   Using Lemma \ref{lem_prop}, it follows that $N_{G}(u_{2})\cap N_{G}(u_{3})=\emptyset$.
   But this contradicts the fact that $v_{1}$ is adjacent to both $u_{2}$ and $u_{3}$.

   \medskip

\noindent{\bf Subcase 1.2.} $u_{2}\nsim v_{3}$.

   \medskip

   Suppose that $N_{G}(u_{2})=\{v_{1},v_{s},v_{t}\}$ with $t>s\geq 4$.
  According to Lemma \ref{lem_adja}, it follows that both $u_{i}\sim v_{s}$ and $u_{i}\sim v_{t}$.
   Suppose that $N_{G}(v_{s})=\{u_{2},u_{i},u_{j}\}$ where $j\geq 3$.
   Since $\mathtt{x}_{u_{1}}>\mathtt{x}_{u_{2}}$ and $\mathtt{x}_{u_{3}}\geq \mathtt{x}_{u_{j}}$,
   we obtain that $$\mathtt{x}[N_{G}(v_{2})]=\mathtt{x}_{u_{1}}+\mathtt{x}_{u_{3}}+\mathtt{x}_{u_{i}}
   >\mathtt{x}_{u_{2}}+\mathtt{x}_{u_{i}}+\mathtt{x}_{u_{j}}=\mathtt{x}[N_{G}(v_{s})].$$
   It follows from Lemma \ref{lem_neig} that $\mathtt{x}_{v_{2}}>\mathtt{x}_{v_{s}}$.
   This implies that $\mathtt{x}_{v_{2}}>\mathtt{x}_{v_{s}}\geq \mathtt{x}_{v_{t}}$.
   If $u_{3}\sim v_{t}$, then $N_{G}(u_{3})=\{v_{1},v_{2},v_{t}\}$.
  Since $\mathtt{x}_{v_{2}}>\mathtt{x}_{v_{s}}$, it follows that
  $$\mathtt{x}[N_{G}(u_{3})]=\mathtt{x}_{v_{1}}+\mathtt{x}_{v_{2}}+\mathtt{x}_{v_{t}}
   >\mathtt{x}_{v_{1}}+\mathtt{x}_{v_{s}}+\mathtt{x}_{v_{t}}=\mathtt{x}[N_{G}(u_{2})].$$
   By Lemma \ref{lem_neig}, we obtain that $\mathtt{x}_{u_{3}}>\mathtt{x}_{u_{2}}$, a contradiction.
   Hence $u_{3}\nsim v_{t}$. Clearly, $\{u_{2}v_{t},u_{3}v_{2}\}$ is a pair of independent edges in $G$.
   Since $\mathtt{x}_{u_{2}}\geq \mathtt{x}_{u_{3}}$ and $\mathtt{x}_{v_{2}}>\mathtt{x}_{v_{t}}$,
   using Lemma \ref{lem_prop}, we obtain that $N_{G}(u_{2})\cap N_{G}(u_{3})=\emptyset$.
   Thus we obtain a contradiction since $v_{1}$ is a common neighbour of $u_{2}$ and $u_{3}$.

   \medskip

\noindent{\bf Case 2.} $v_{2}\nsim u_{3}$.

 \medskip

 Arguing similarly to that in Subcase 1.2, we can also obtain a contradiction if $u_{2}\sim v_{3}$.
 So we assume that $u_{2}\nsim v_{3}$.
 Suppose that $N_{G}(v_{2})=\{u_{1},u_{i},u_{j}\}$ and $N_{G}(u_{2})=\{v_{1},v_{s},v_{t}\}$, where $j>i\geq 4$ and $t>s\geq 4$.
 We claim that $\mathtt{x}_{u_{2}}>\mathtt{x}_{u_{j}}$. If not, it follows that $\mathtt{x}_{u_{2}}\leq \mathtt{x}_{u_{j}}\leq \mathtt{x}_{u_{i}}$.
 Thus $\mathtt{x}[N_{G}(v_{2})]\geq \mathtt{x}_{u_{1}}+2\mathtt{x}_{u_{2}}\geq \mathtt{x}[N_{G}(v_{1})]$.
 By Lemma \ref{lem_neig}, we obtain $\mathtt{x}_{v_{2}}\geq \mathtt{x}_{v_{2}}$, contradicting Claim 1.
 This implies that $\mathtt{x}_{u_{2}}>\mathtt{x}_{u_{j}}$.
 Similarly, one can also obtain $\mathtt{x}_{v_{2}}>\mathtt{x}_{v_{t}}$.

 Since $\mathtt{x}_{u_{2}}>\mathtt{x}_{u_{j}}$ and $\mathtt{x}_{v_{2}}\geq \mathtt{x}_{v_{s}}\geq \mathtt{x}_{v_{t}}$,
 it follows from Lemma \ref{lem_adja} that $u_{j}\sim v_{s}$ and $u_{j}\sim v_{t}$.
 Since $\mathtt{x}_{v_{2}}>\mathtt{x}_{v_{t}}$, by Lemma \ref{lem_adja}, we obtain that $v_{t}\sim u_{i}$.

 We claim that $u_{i}\sim v_{s}$.
 Suppose that $u_{i}\nsim v_{s}$. If $\mathtt{x}_{u_{2}}>\mathtt{x}_{u_{i}}$, then it follows from Lemma \ref{lem_adja} that $u_{i}\sim v_{s}$, a contradiction.
 Hence $\mathtt{x}_{u_{2}}=\mathtt{x}_{u_{i}}$. Similarly, one can see that $\mathtt{x}_{v_{2}}=\mathtt{x}_{v_{s}}$.
 It follows that
 $$\mathtt{x}[N_{G}(v_{2})]=\mathtt{x}_{u_{1}}+\mathtt{x}_{u_i}+\mathtt{x}_{u_{j}}
 >2\mathtt{x}_{u_{2}}+\mathtt{x}_{u_{j}}\geq \mathtt{x}[N_{G}(v_{s})].$$
By Lemma \ref{lem_neig}, we have $\mathtt{x}_{v_{2}}>\mathtt{x}_{v_{s}}$, a contradiction.
Therefore $u_{i}\sim v_{s}$.

Suppose now without loss of generality that $\mathtt{x}_{u_{1}}\geq \mathtt{x}_{v_{1}}$.
Lemma \ref{lem_pos} shows that $\mathtt{x}_{u_{1}}>0$ and $\mathtt{x}_{v_{1}}>0$.
Since $\mathtt{x}_{v_{1}}>\mathtt{x}_{v_{2}}$, we can choose a real number $\varepsilon>0$ such that $\mathtt{x}_{v_{1}}>\varepsilon>\mathtt{x}_{v_{2}}$.
It follows from Lemma \ref{lem_con} that $[V(G)]_{\leq \varepsilon}$ is connected.
On the other hand, according to the above adjacency relations, it is easy to see that $\{u_{1},v_{1}\}$
is a vertex cut of $G$, and the vertices $v_{2}$ and $v_{3}$ belong to the different components in $G-\{u_{1},v_{1}\}$.
This implies that any path from $v_{2}$ to $v_{3}$ must pass through either $u_{1}$ or $v_{1}$.
Note that $\{v_{2},v_{3}\}\subseteq [V(G)]_{\leq \varepsilon}$ and $\{u_{1},v_{1}\}\cap [V(G)]_{\leq \varepsilon}=\emptyset$,
as $\mathtt{x}_{u_{1}}\geq \mathtt{x}_{v_{1}}>\varepsilon>\mathtt{x}_{v_{2}}\geq \mathtt{x}_{v_{3}}$.
Then we obtain that there are no paths from $v_{2}$ to $v_{3}$ in the subgraph induced by $[V(G)]_{\leq \varepsilon}$,
a contradiction to the connectedness of $[V(G)]_{\leq \varepsilon}$.
The proof is completed.
\end{proof}

\begin{lemma}\label{lem_v2u3}
$v_{2}\sim u_{3}$.
\end{lemma}

\begin{proof}
  Suppose to the contrary that $v_{2}\nsim u_{3}$. Let $N_{G}(v_{2})=\{u_{1},u_{2},u_{i}\}$ with $i\geq 4$.
  To obtain the contradiction, we will consider the following cases.

  \medskip

\noindent{\bf Case 1.} $u_{2}\sim v_{3}$.

  \medskip

  It is easy to see that $v_{3}\nsim u_{3}$. If not, $v_{3}$ and $v_{1}$ have the same neighbours,
  thus Lemma \ref{lem_neig} leads to that $\mathtt{x}_{v_{1}}=\mathtt{x}_{v_{3}}$, contradicting Lemma \ref{lem_neq}.
  Suppose that $N_{G}(u_{3})=\{v_{1},v_{s},v_{t}\}$, where $t>s\geq 4$.
  We claim that $u_{i}$ is adjacent to both $v_{s}$ and $v_{t}$. Note that $\mathtt{x}_{u_{3}}\geq \mathtt{x}_{u_{i}}$
  and $\mathtt{x}_{v_{2}}\geq \mathtt{x}_{v_{s}}\geq \mathtt{x}_{v_{t}}$.
  If $\mathtt{x}_{u_{3}}>\mathtt{x}_{u_{i}}$, then Lemma \ref{lem_adja} shows that $u_{i}$ is adjacent to both $v_{s}$ and $v_{t}$.
  If $\mathtt{x}_{u_{3}}=\mathtt{x}_{u_{i}}$, then
  $$\mathtt{x}[N_{G}(v_{s})]\leq 3\mathtt{x}_{u_{3}}
  <\mathtt{x}_{u_{1}}+\mathtt{x}_{u_{2}}+\mathtt{x}_{u_{i}}=\mathtt{x}[N_{G}(v_{2})].$$
  Then it follows from Lemma \ref{lem_neig} that $\mathtt{x}_{v_{2}}>\mathtt{x}_{v_{s}}\geq \mathtt{x}_{v_{t}}$.
  Again, Lemma \ref{lem_adja} shows that $u_{i}$ is adjacent to both $v_{s}$ and $v_{t}$. Hence $v_{3}\nsim u_{i}$.
  Suppose that $N_{G}(v_{3})=\{u_{1},u_{2},u_{j}\}$ where $j\neq i$ and $j\geq 4$.
  Since $\mathtt{x}_{v_{1}}>\mathtt{x}_{v_{3}}$, by Lemma \ref{lem_neig}, we have $\mathtt{x}[N_{G}(v_{1})]>\mathtt{x}[N_{G}(v_{3})]$,
  hence $\mathtt{x}_{u_{3}}>\mathtt{x}_{u_{j}}$. Note also that $\mathtt{x}_{v_{3}}\geq \mathtt{x}_{v_{s}}\geq \mathtt{x}_{v_{t}}$.
  Using Lemma \ref{lem_adja}, we obtain that $u_{j}\sim v_{s}$ and $u_{j}\sim v_{t}$.
  Thus the subgraph of $G$ induced by $\{u_{1},u_{2},u_{3},u_{i},u_{j},v_{1},v_{2},v_{3},v_{s},v_{t}\}$ is a cubic bipartite graph, a contradiction.

  \medskip

\noindent{\bf Case 2.} $u_{2}\nsim v_{3}$.

  \medskip

  Suppose that $N_{G}(u_{2})=\{v_{1},v_{2},v_{s}\}$ with $s\geq 4$.

  \medskip

\noindent{\bf Subcase 2.1.} $u_{3}\sim v_{3}$.

  \medskip

  We first prove that $v_{3}\sim u_{i}$. Suppose to the contrary that $v_{3}\nsim u_{i}$.
  Then $\{u_{3}v_{3},u_{i}v_{2}\}$ forms a pair of independent edges in $G$.
  Clearly, $\mathtt{x}_{u_{3}}\geq \mathtt{x}_{u_{i}}$ and $\mathtt{x}_{v_{2}}\geq \mathtt{x}_{v_{3}}$.
  If either $\mathtt{x}_{u_{3}}\neq \mathtt{x}_{u_{i}}$ or $\mathtt{x}_{v_{2}}\neq \mathtt{x}_{v_{3}}$,
  then it follows form Lemma \ref{lem_prop} (P3) that $N_{G}(v_{2})\cap N_{G}(v_{3})=\emptyset$,
  which contradicts the fact that $u_{1}\in N_{G}(v_{2})\cap N_{G}(v_{3})$.
  Hence $\mathtt{x}_{u_{3}}=\mathtt{x}_{u_{i}}$ and $\mathtt{x}_{v_{2}}=\mathtt{x}_{v_{3}}$.
  Since $\mathtt{x}_{u_{3}}=\mathtt{x}_{u_{i}}$, it follows from Lemma \ref{lem_neig} that $\mathtt{x}_{v_{1}}=\mathtt{x}_{v_{2}}=\mathtt{x}_{v_{3}}$,
  a contradiction with Lemma \ref{lem_neq}.
  Hence $v_{3}\sim u_{i}$.
  Similarly, one can obtain that $u_{3}\sim v_{s}$.

  If $u_{i}\sim v_{s}$, then the subgraph of $G$ induced by $\{u_{1},u_{2},u_{3},u_{i},v_{1},v_{2},v_{3},v_{s}\}$ is a cubic bipartite graph, a contradiction.
  This implies that $u_{i}\nsim v_{s}$. Obviously, $\{u_{2}v_{s},u_{i}v_{3}\}$ is a pair of independent edges in $G$.
  Note that $\mathtt{x}[N_{G}(v_{1})]=\mathtt{x}_{u_{1}}+\mathtt{x}_{u_{2}}+\mathtt{x}_{u_{3}}$
  and $\mathtt{x}[N_{G}(v_{3})]=\mathtt{x}_{u_{1}}+\mathtt{x}_{u_{3}}+\mathtt{x}_{u_{i}}$.
  Since $\mathtt{x}_{v_{1}}>\mathtt{x}_{v_{3}}$, it follows from Lemma \ref{lem_neig} that $\mathtt{x}_{u_{2}}>\mathtt{x}_{u_{i}}$.
  Using Lemma \ref{lem_prop} (P3), we have $N_{G}(v_{3})\cap N_{G}(v_{s})=\emptyset$.
  But $u_{3}$ is a common neighbour of $v_{3}$ and $v_{s}$, a contradiction again.

  \medskip

\noindent{\bf Subcase 2.2.} $u_{3}\nsim v_{3}$.

  \medskip

  Let $N_{G}(u_{3})=\{v_{1},v_{s},v_{t}\}$ with $t>s\geq 4$.
  We will show that $u_{i}$ is adjacent to both $v_{s}$ and $v_{t}$.
  Note that $\mathtt{x}_{u_{3}}\geq \mathtt{x}_{u_{i}}$ and $\mathtt{x}_{v_{2}}\geq \mathtt{x}_{v_{s}}\geq \mathtt{x}_{v_{t}}$.
  If $\mathtt{x}_{u_{3}}>\mathtt{x}_{u_{i}}$, then Lemma \ref{lem_adja} shows that $u_{i}\sim v_{s}$ and $u_{i}\sim v_{t}$.
  Assume that $\mathtt{x}_{u_{3}}=\mathtt{x}_{u_{i}}$.
  Then we have
  $$\mathtt{x}[N_{G}(v_{1})]=\mathtt{x}_{u_{1}}+\mathtt{x}_{u_{2}}+\mathtt{x}_{u_{3}}
  =\mathtt{x}_{u_{1}}+\mathtt{x}_{u_{2}}+\mathtt{x}_{u_{i}}=\mathtt{x}[N_{G}(v_{2})].$$
  According to Lemma \ref{lem_neig},
  it follows that $\mathtt{x}_{v_{1}}=\mathtt{x}_{v_{2}}$. Combining with Lemma \ref{lem_neq}, we obtain that $\mathtt{x}_{v_{2}}>\mathtt{x}_{v_{3}}$,
  hence $\mathtt{x}_{v_{2}}>\mathtt{x}_{v_{s}}\geq \mathtt{x}_{v_{t}}$.
  Using Lemma \ref{lem_adja}, we also obtain that $u_{i}\sim v_{s}$ and $u_{i}\sim v_{t}$.

  We may assume that $N_{G}(v_{3})=\{u_{1},u_{j},u_{k}\}$ with $k>j\geq 4$. Clearly, $i\neq j$ and $i\neq k$.
  Note that $\mathtt{x}_{u_{3}}\geq \mathtt{x}_{u_{j}}\geq \mathtt{x}_{u_{k}}$ and $\mathtt{x}_{v_{3}}\geq \mathtt{x}_{v_{t}}$.
  If $\mathtt{x}_{v_{3}}\geq \mathtt{x}_{v_{t}}$, then Lemma \ref{lem_adja} shows that $v_{t}\sim u_{j}$ and $u_{t}\sim u_{k}$.
  Suppose now that $\mathtt{x}_{v_{3}}=\mathtt{x}_{v_{t}}$.
  Hence $$\mathtt{x}[N_{G}(u_{3})]=\mathtt{x}_{v_{1}}+\mathtt{x}_{v_{s}}+\mathtt{x}_{v_{t}}
  =\mathtt{x}_{v_{1}}+2\mathtt{x}_{v_{3}}>3\mathtt{x}_{v_{3}}\geq \mathtt{x}[N_{G}(u_{j})].$$
  It follows from Lemma \ref{lem_neig} that $\mathtt{x}_{u_{3}}>\mathtt{x}_{u_{j}}$.
  According to Lemma \ref{lem_adja}, one can see that $v_{t}\sim u_{j}$ and $u_{t}\sim u_{k}$.

  Therefore, $\{u_{3},u_{i},u_{j},u_{k}\}\subseteq N_{G}(v_{t})$. But this contradicts the degree of $v_{t}$.
\end{proof}

\begin{lemma}\label{lem_u2v3}
  $u_{2}\sim v_{3}$.
\end{lemma}

\begin{proof}
    By contradiction. Suppose that $u_{2}\nsim v_{3}$. Lemma \ref{lem_v2u3} shows that $v_{2}\sim u_{3}$.
    After interchanging the signs $u$ and $v$, one can obtain a contradiction by the same argument as in Case 1 of Lemma  \ref{lem_v2u3}.
  \end{proof}

  \begin{lemma}
    $\mathtt{x}_{u_{1}}=\mathtt{x}_{u_{2}}>\mathtt{x}_{u_{3}}$ and $\mathtt{x}_{v_{1}}=\mathtt{x}_{v_{2}}>\mathtt{x}_{v_{3}}$.
  \end{lemma}

  \begin{proof}
    Combining Lemmas \ref{lem_u2v2}, \ref{lem_v2u3} and \ref{lem_u2v3},
    one can see that the adjacency relations of the vertices $\{u_{1},u_{2},u_{3},v_{1},v_{2},v_{3}\}$ are as presented in Figure \ref{fig_main}.
    Note that $u_{1}$ and $u_{2}$ have the same neighbours. According to Lemma \ref{lem_neig}, it follows that $\mathtt{x}_{u_{1}}=\mathtt{x}_{u_{2}}$.
    By Lemma \ref{lem_neq} we obtain $\mathtt{x}_{u_{1}}=\mathtt{x}_{u_{2}}>\mathtt{x}_{u_{3}}$.
    Similarly, we can see that $\mathtt{x}_{v_{1}}=\mathtt{x}_{v_{2}}>\mathtt{x}_{v_{3}}$.
  \end{proof}

  \begin{lemma}\label{lem_v3u4}
    $v_{3}\sim u_{4}$ and $\mathtt{x}_{u_{4}}>\mathtt{x}_{u_{5}}$.
  \end{lemma}
  \begin{proof}
    Now there is only one undetermined neighbour of $v_{3}$. We may assume that $N_{G}(v_{3})=\{u_{1},u_{2},u_{i}\}$ with $i\geq 4$.
    It suffices to show that $\mathtt{x}_{u_{i}}>\max\{\mathtt{x}_{u_{j}}:j\geq 4~\text{and}~j\neq i\}$.

    Suppose to the contrary that $\mathtt{x}_{u_{j}}\geq \mathtt{x}_{u_{i}}$ where $j\geq 4$ and $j\neq i$.
    Clearly, $u_{j}\nsim v_{3}$. Assume that $N_{G}(u_{j})=\{v_{s},v_{t},v_{k}\}$ with $k>t>s\geq 4$.
    If $\mathtt{x}_{u_{j}}>\mathtt{x}_{u_{i}}$, then it follows from Lemma \ref{lem_adja} that $\{v_{s},v_{t},v_{k}\}\subseteq N_{G}(u_{i})$.
    However, this contradicts the fact that $G$ is cubic. Suppose now that $\mathtt{x}_{u_{j}}=\mathtt{x}_{u_{i}}$.
    Hence $$\mathtt{x}[N_{G}(v_{3})]=\mathtt{x}_{u_{1}}+\mathtt{x}_{u_{2}}+\mathtt{x}_{u_{i}}
    >2\mathtt{x}_{u_{3}}+\mathtt{x}_{u_{j}}\geq \mathtt{x}[N_{G}(v_{s})].$$
    According to Lemma \ref{lem_neig}, we have $\mathtt{x}_{v_{3}}>\mathtt{x}_{v_{s}}$,
    that is, $\mathtt{x}_{v_{3}}>\mathtt{x}_{v_{s}}\geq \mathtt{x}_{v_{t}}\geq \mathtt{x}_{v_{k}}$.
    Using Lemma \ref{lem_adja}, we also obtain that $\{v_{s},v_{t},v_{k}\}\subseteq N_{G}(u_{i})$, a contradiction again.
  \end{proof}

  Therefore, arguing similarly to that in the proof of Lemma \ref{lem_v3u4}, we obtain the following lemma.

  \begin{lemma}\label{lem_u3v4}
    $u_{3}\sim v_{4}$ and $\mathtt{x}_{v_{4}}>\mathtt{x}_{v_{5}}$.
  \end{lemma}

  Let us consider the adjacency relations for the subsequent vertices. For a given positive integer $k$, we define
  $$[u]_{k}=\{u_{i}:1\leq i\leq k\} ~~\text{and}~~ [v]_{k}=\{v_{i}:1\leq i\leq k\}.$$

  \begin{lemma}\label{lem_gen}
    For any $4\leq i\leq n-3$, the following statements hold:\\
    \noindent {\rm (1)} $u_{i}\sim v_{i-1}$, $v_{i}\sim u_{i-1}$ and $v_{i}\sim u_{i}$,\\
  \noindent {\rm (2)} $\mathtt{x}_{v_{i}}>\mathtt{x}_{v_{i+1}}$ and $\mathtt{x}_{u_{i}}>\mathtt{x}_{u_{i+1}}$.
  \end{lemma}

  \begin{proof}
    We prove the result by induction on $i$. Consider the case $i=4$. Lemmas \ref{lem_v3u4} and \ref{lem_u3v4} show that
    \begin{eqnarray*}
      \left\{
      \begin{aligned}
        &v_{3}\sim u_{4},\\
        &u_{3}\sim v_{4},
      \end{aligned}
    \right.~~\text{and}~~
    \left\{
      \begin{aligned}
        &\mathtt{x}_{u_{4}}>\mathtt{x}_{u_{5}},\\
        &\mathtt{x}_{v_{4}}>\mathtt{x}_{v_{5}}.
      \end{aligned}
    \right.
    \end{eqnarray*}
    It suffices to show that $u_{4}\sim v_{4}$. Suppose to the contrary that $u_{4}\nsim v_{4}$.
    Note that $$\mathtt{x}_{u_{4}}>\max\{\mathtt{x}_{u}:u\in N_{G}(v_{4})\backslash\{u_{3}\}\}
    ~~\text{and}~~\mathtt{x}_{v_{4}}>\max\{\mathtt{x}_{v}:v\in N_{G}(u_{4})\backslash\{v_{3}\}\}.$$
    According to Lemma \ref{lem_adja}, it is easy to see that every vertex in $N_{G}(v_{4})\backslash\{u_{3}\}$
    is adjacent to every vertex in $N_{G}(u_{4})\backslash\{v_{3}\}$.
    Therefore, the subgraph of $G$ induced by $[u]_{4}\cup [v]_{4}\cup N_{G}(u_{4})\backslash\{v_{3}\}\cup N_{G}(v_{4})\backslash\{u_{3}\}$
    is a cubic bipartite graph. This implies that $\{v_{j}:5\leq j\leq n\}\cup \{u_{j}:5\leq j\leq n\}$ contains only four vertices.
    But this contradicts the fact $4\leq i\leq n-3$. Hence $u_{4}\sim v_{4}$.
    The result holds for $i=4$.

    Suppose now that $i\geq 5$. By the induction hypothesis, we obtain that
    \begin{eqnarray*}
      \left\{
      \begin{aligned}
        &v_{i-1}\sim u_{i-1},\\
        &u_{i-1}\sim v_{i-2},\\
        &v_{i-1}\sim u_{i-2},
      \end{aligned}
    \right.~~\text{and}~~
    \left\{
      \begin{aligned}
        &\mathtt{x}_{u_{i-1}}>\mathtt{x}_{u_{i}},\\
        &\mathtt{x}_{v_{i-1}}>\mathtt{x}_{v_{i}}.
      \end{aligned}
    \right.
    \end{eqnarray*}
    We first prove that $u_{i}\sim v_{i-1}$ and $\mathtt{x}_{u_{i}}>\mathtt{x}_{u_{i+1}}$.
    Let $u_{j}$ be the neighbour of $v_{i-1}$ with $j\geq i$. It suffices to prove that $\mathtt{x}_{u_{j}}>\max\{\mathtt{x}_{u_{k}}:k\geq i~\text{and}~k\neq j\}$.
    Suppose to the contrary that there is a vertex $u_{k}$ such that $\mathtt{x}_{u_{k}}\geq \mathtt{x}_{u_{j}}$, where $k\geq i$ and $k\neq j$.
    In this case, it is to see that $N_{G}(u_{k})\subseteq [v]_{n}\backslash [v]_{i-1}$.
    Clearly, $\mathtt{x}_{v_{i-1}}>\max\{\mathtt{x}_{v}:v\in [v]_{n}\backslash [v]_{i-1}\}$,
    which implies that $\mathtt{x}_{v_{i-1}}>\max\{\mathtt{x}_{v}:v\in N_{G}(u_{k})\}$.
    By Lemma \ref{lem_adja}, the vertex $u_{j}$ is adjacent to all vertices in $N_{G}(u_{k})$, and hence the degree of $u_{j}$ is at least 4, a contradiction.
    Therefore, it follows that $u_{i}\sim v_{i-1}$ and $\mathtt{x}_{u_{i}}>\mathtt{x}_{u_{i+1}}$.
    Similarly, we can show that  $v_{i}\sim u_{i-1}$ and $\mathtt{x}_{v_{i}}>\mathtt{x}_{v_{i+1}}$.

    It remains to prove that $u_{i}\sim v_{i}$. Assume that $u_{i}\nsim v_{i}$. Then one can see that
    $N_{G}(u_{i})\backslash\{v_{i-1}\}\subseteq [v]_{n}\backslash [v]_{i}$, $N_{G}(v_{i})\backslash\{u_{i-1}\}\subseteq [u]_{n}\backslash [u]_{i}$
    and $|N_{G}(u_{i})\backslash\{v_{i-1}\}|=|N_{G}(v_{i})\backslash\{u_{i-1}\}|=2$.
    Since $\mathtt{x}_{u_{i}}>\mathtt{x}_{u_{i+1}}$ and $\mathtt{x}_{v_{i}}>\mathtt{x}_{v_{i+1}}$,
    by Lemma \ref{lem_adja}, every vertex in $N_{G}(u_{i})\backslash\{v_{i-1}\}$ is adjacent to every vertex in $N_{G}(v_{i})\backslash\{u_{i-1}\}$.
    It is easy to see that the subgraph of $G$ induced by $[u]_{i}\cup [v]_{i}\cup N_{G}(u_{i})\backslash\{v_{i-1}\}\cup N_{G}(v_{i})\backslash\{u_{i-1}\}$
    is a cubic bipartite graph.
    This implies that there are only four vertices in $([u]_{n}\backslash[u]_{i})\cup ([v]_{n}\backslash[v]_{i})$, which yields that $i=n-2$.
    But this contradicts the fact $i\leq n-3$. Hence $u_{i}\sim v_{i}$ and the proof is completed.
  \end{proof}

In summary, according to Lemmas \ref{lem_u2v2}, \ref{lem_v2u3}, \ref{lem_u2v3} and \ref{lem_gen},
we determine the adjacency relations of all vertices in $[u]_{n-3}\cup [v]_{n-3}$.
Now we are ready to prove Theorem \ref{the_extr}.

\begin{proof}[\bf Proof of Theorem \ref{the_extr}]
  If $v_{n-3}\nsim u_{n-2}$, then $v_{n-3}$ must be adjacent to a vertex $u\in \{u_{n-1},u_{n}\}$.
  Hence $N_{G}(u_{n-2})=\{v_{n-2},v_{n-1},v_{n}\}$.
  Lemma \ref{lem_gen} shows that $\mathtt{x}_{v_{n-3}}>\mathtt{x}_{v_{n-2}}\geq \mathtt{x}_{v_{n-1}}\geq \mathtt{x}_{v_{n}}$.
  Morevoer, since $\mathtt{x}_{u_{n-2}}\geq \mathtt{x}_{u}$, by Lemma \ref{lem_adja}, the vertex $u$ is adjacent to all vertices in $N_{G}(u_{n-2})$.
  This implies that the degree of $u$ is equal to 4, a contradiction. Hence $v_{n-3}\sim u_{n-2}$.
  Similarly, one can see that $u_{n-3}\sim v_{n-2}$.

  When $u_{n-2}\sim v_{n-2}$, the graph $G$ cannot be a cubic bipartite graph.
  If $u_{n-2}\nsim v_{n-2}$, it is easy to see that $G$ is clearly isomorphic to $H_{2n}$, and the result follows.
\end{proof}

\section{Algebraic connectivity of the extremal graph}\label{ex_value}

In this section, we consider the asymptotic value for the algebraic connectivity of the extremal graph.
As mentioned in Section \ref{ex_graph}, $H_{2n}$ is the unique cubic bipartite graph with minimum algebraic connectivity.
Label the vertices of $H_{2n}$ as shown in Figure \ref{fig_main}.
Before proceeding further, let us present the following property for the Fiedler vector of $H_{2n}.$

\begin{lemma}\label{lem_sym}
The extremal graph $H_{2n}$ has a Fiedler vector $\mathtt{z}$ such that $\mathtt{z}_{u_{i}}=\mathtt{z}_{v_{i}}$ for any $1\leq i\leq n$.
\end{lemma}
\begin{proof}
According to Lemma \ref{lem_pos}, for the extremal graph $H_{2n}$,
there is a Fiedler vector $\mathtt{x}$ such that $\mathtt{x}_{u_{1}}>0$ and $\mathtt{x}_{v_{1}}>0$.
Let $a_{i}=\mathtt{x}_{u_{i}}$ and $b_{i}=\mathtt{x}_{v_{i}}$, where $1\leq i\leq n$.
Construct a vector $\mathtt{y}$ on the vertices of $H_{2n}$ such that $\mathtt{y}_{u_{i}}=b_{i}$ and $\mathtt{y}_{v_{i}}=a_{i}$ for $1\leq i\leq n$.
According to the symmetry of $H_{2n}$, it is easy to see that $\mathtt{y}$ is also a Fiedler vector of $H_{2n}$.
Let us consider a new vector $\mathtt{z}=\mathtt{x}+\mathtt{y}$.
Since $\mathtt{x}\bot \mathtt{1}$ and $\mathtt{y}\bot \mathtt{1}$, then we obtain that $\mathtt{z}\bot \mathtt{1}$.
Clearly, $\mathtt{z}$ is a nonzero vector since $\mathtt{z}_{u_{1}}=\mathtt{x}_{u_{1}}+\mathtt{x}_{v_{1}}$.
Hence $\mathtt{z}$ is also a Fiedler vector of $H_{2n}$.
Moreover, one can see that $\mathtt{z}_{u_{i}}=a_{i}+b_{i}=\mathtt{z}_{v_{i}}$ for $1\leq i\leq n$.
Obviously $\mathtt{z}$ is a Fiedler vector of $H_{2n}$ with the desired property.
\end{proof}

Let $P_{n}$ denote a path on $n$ vertices. The Laplacian spectrum of $P_{n}$ consists of the numbers $2-2\cos(\pi j/n)$, where $0\leq j\leq n-1$ (see, e.g., \cite{Brouwer2012}). Clearly, $a(P_{n})=2-2\cos(\pi/n)$. Note also that $L(P_{n})$ is a tridiagonal matrix.
The eigenvalues and eigenvectors of a certain tridiagonal matrix were discussed in \cite{Willms2008}.
Indeed, applying the conclusions in \cite{Willms2008}, one can also derive all eigenvalues and eigenvectors of $L(P_{n})$.
In particular, the Fiedler vector of $P_{n}$ is skew symmetry. To make the paper self-contained, we give a proof by using Lemma \ref{lem_con}.

\begin{lemma}\label{lem_path}
  The Fiedler vector of $P_{n}$ is skew symmetry.
\end{lemma}
\begin{proof}
  Let $\mathtt{x}$ be a unit Fiedler vector of $P_{n}$. Suppose that $P_{n}=v_{1}v_{2}\cdots v_{n}$.
  Without loss of generality, we may assume that $\mathtt{x}_{v_{1}}\geq 0$.

  We first claim that $\mathtt{x}_{v_{1}}>0$. Suppose that $\mathtt{x}_{v_{1}}=0$.
  Since $a(P_{n})\mathtt{x}_{v_{1}}=\mathtt{x}_{v_{1}}-\mathtt{x}_{v_{2}}$, then we have $\mathtt{x}_{v_{2}}=0$.
  Furthermore, one can see that $\mathtt{x}_{v_{3}}=0$ since $a(P_{n})\mathtt{x}_{v_{2}}=2\mathtt{x}_{v_{2}}-\mathtt{x}_{v_{1}}-\mathtt{x}_{v_{3}}$.
  Repeating this process, we will finally obtain a zero vector, a contradiction.

  Let us consider another vector $\mathtt{y}$ on the vertices of $P_{n}$, where $\mathtt{y}_{v_{i}}=\mathtt{x}_{v_{n+1-i}}$ for $1\leq i\leq n$.
  Clearly, the vector $\mathtt{y}$ is obtained from $\mathtt{x}$ by exchanging the entries corresponding to vertices $v_{i}$ and $v_{n+1-i}$.
  According to the symmetry of $P_{n}$, $\mathtt{y}$ is also a Fiedler vector of $P_{n}$.
  Recall that $a(P_{n})$ is a simple eigenvalue of $L(G)$. Hence $\mathtt{y}=\pm\mathtt{x}$.

  If $\mathtt{y}=\mathtt{x}$, then $\mathtt{x}_{v_{1}}=\mathtt{x}_{v_{n}}$.
  Since $\mathtt{x}\bot\mathtt{1}$, then the vector contains a negative entry.
  Suppose that $\mathtt{x}_{v_{i}}<0$ with $1<i<n$.
  Let $[V(P_{n})]_{\geq 0}=\{v\in V(P_{n}):\mathtt{x}_{v}\geq 0\}$.
  Lemma \ref{lem_con} shows that the subgraph induced by $[V(P_{n})]_{\geq 0}$ is connected.
  But this is impossible since $v_{i}\notin [V(P_{n})]_{\geq 0}$.
  It follows that $\mathtt{y}=-\mathtt{x}$. Hence $\mathtt{x}_{v_{i}}=-\mathtt{x}_{v_{n+1-i}}$ for $1\leq i\leq n$, as required.
\end{proof}

\begin{theorem}\label{the_H}
The algebraic connectivity of $H_{2n}$ is $a(H_{2n})=(1+o(1))\frac{\pi^{2}}{n^{2}}$.
\end{theorem}

\begin{proof}
We first prove that $(1+o(1))\frac{\pi^{2}}{n^{2}}$ is a lower bound for $a(H_{2n})$.
Let $\mathtt{z}$ be a Fiedler vector of $H_{2n}$ satisfying the property in Lemma \ref{lem_sym}.
Hence $\mathtt{z}_{u_{i}}=\mathtt{z}_{v_{i}}$ for any $1\leq i\leq n$.
Note that
$$a(H_{2n})=\frac{\mathtt{z}^{t}L(H_{2n})\mathtt{z}}{\mathtt{z}^{t}\mathtt{z}}.$$
Expanding the right side of the above equality, we have
\begin{eqnarray}\label{eq3}
  a(H_{2n})= \frac{2(\mathtt{z}_{v_{1}}-\mathtt{z}_{v_{3}})^{2}+2(\mathtt{z}_{v_{n-2}}-\mathtt{z}_{v_{n}})^{2}
  +2\sum_{i=1}^{n-1}(\mathtt{z}_{v_{i}}-\mathtt{z}_{v_{i+1}})^{2}}{2\sum_{i=1}^{n}\mathtt{z}_{v_{i}}^{2}}
  \geq \frac{\sum_{i=1}^{n-1}(\mathtt{z}_{v_{i}}-\mathtt{z}_{v_{i+1}})^{2}}{\sum_{i=1}^{n}\mathtt{z}_{v_{i}}^{2}}.
\end{eqnarray}
Let $P_{n}$ be a path on $n$ vertices. According to (\ref{eq1}), it follows that
\begin{eqnarray}\label{eq4}
  a(P_{n})\leq \frac{\sum_{i=1}^{n-1}(\mathtt{z}_{v_{i}}-\mathtt{z}_{v_{i+1}})^{2}}{\sum_{i=1}^{n}\mathtt{z}_{v_{i}}^{2}}.
\end{eqnarray}
Combining (\ref{eq3}) and (\ref{eq4}), we obtain $a(H_{2n})\geq a(P_{n})$.
It is well-known that $a(P_{n})=2-2\cos (\pi/n)$.
Thus we obtain that
$$a(H_{2n})\geq 2-2\cos\left( \frac{\pi}{n}\right)=4\sin^{2}\left( \frac{\pi}{2n}\right)= (1+o(1))\frac{\pi^{2}}{n^{2}},$$
which implies the desired lower bound.

Now we show that $(1+o(1))\frac{\pi^{2}}{n^{2}}$ is an upper bound for $a(H_{2n})$.
Let $P_{n-4}=w_{1}w_{2}\cdots w_{n-4}$ be a path on $n-4$ vertices.
Assume that $\mathtt{x}$ is a Fiedler vector of $P_{n-4}$.
Set $\mathtt{x}_{w_{i}}=a_{i}$ for $1\leq i\leq n-4$.
Construct a new vector $\mathtt{y}$ on vertices of $H_{2n}$ as follows:
 \begin{eqnarray*}
     \mathtt{y}_{u_{i}}=\mathtt{y}_{v_{i}}= \left\{
      \begin{aligned}
    & a_{1}, && \text{for}~~1\leq i\leq 3,\\
    &  a_{i-2}, & &\text{for}~~4\leq i\leq n-3,\\
    &  a_{n-4}, && \text{for}~~n-2\leq i\leq n.
      \end{aligned}
    \right.
    \end{eqnarray*}
According to Lemma \ref{lem_path}, it is easy to see that $\mathtt{y}\bot\mathtt{1}$.
By (\ref{eq1}), we obtain that
\begin{eqnarray}\label{eq5}
a(H_{2n})\leq \frac{\mathtt{y}^{t}L(H_{2n})\mathtt{y}}{\mathtt{y}^{t}\mathtt{y}}
=\frac{2\sum_{i=1}^{n-5}(a_{i}-a_{i+1})^{2}}{4a_{1}^{2}+4a_{n-4}^{2}+2\sum_{i=1}^{n-4}a_{i}^{2}}
\leq \frac{\sum_{i=1}^{n-5}(a_{i}-a_{i+1})^{2}}{\sum_{i=1}^{n-4}a_{i}^{2}}.
\end{eqnarray}
On the other hand, since $\mathtt{x}$ is a Fiedler vector of $P_{n-4}$, it follows that
\begin{eqnarray}\label{eq6}
a(P_{n-4})=\frac{\mathtt{x}^{t}L(P_{n-4})\mathtt{x}}{\mathtt{x}^{t}\mathtt{x}}=\frac{\sum_{i=1}^{n-5}(a_{i}-a_{i+1})^{2}}{\sum_{i=1}^{n-4}a_{i}^{2}}.
\end{eqnarray}
Combining (\ref{eq5}) and (\ref{eq6}), we obtain
$$a(H_{2n})\leq a(P_{n-4})=2-2\cos\left( \frac{\pi}{n-4}\right)=(1+o(1))\frac{\pi^{2}}{n^{2}}.$$
Therefore, $a(H_{2n})=(1+o(1))\frac{\pi^{2}}{n^{2}}$.
\end{proof}

Combining Theorems \ref{the_extr} and \ref{the_H}, we obtain Theorem \ref{the_min_val} immediately.

\section{Concluding remarks}

In this paper, we determine the unique graph with minimum algebraic connectivity among all connected cubic bipartite graphs,
and obtain the asymptotic value of the minimum algebraic connectivity. For any $k$-regular graph $G$, its diagonal martix of vertex degrees is $D(G)=kI$,
where $I$ is the identity matrix. According to the definition of Laplacian matrix, one can see that the Laplacian matrix and adjacency matrix of $G$ satisfy
$$L(G)=kI-A(G).$$
The difference between the two largest eigenvalues of $A(G)$ is called the spectral gap of $G$. According to the above equation, one can see that
the spectral gap of $G$ is equal to its algebraic connectivity. Using this fact, Theorems \ref{the_min_val} and \ref{the_extr} imply the following results directly.

\begin{theorem}
  The minimum spectral gap of connected cubic bipartite graphs on $2n$ vertices is $(1+o(1))\frac{\pi^{2}}{n^{2}}$.
\end{theorem}

\begin{theorem}
  Among all connected cubic bipartite graphs with at least 12 vertices, $H_{2n}$ is the unique graph with minimum spectral gap.
\end{theorem}

Lemma \ref{lem_trans} establishes a transformation which strictly decreases the algebraic connectivity.
We remark that this transformation also works for general $k$-regular bipartite graphs.
Then it is natural to investigate the minimum algebraic connectivity (spectral gap) of $k$-regular bipartite graphs for $k\geq 4$.
Based on our observations, we think that the extremal graph will be path-like when $n$ is sufficiently large.
In Theorem \ref{the_equ}, we present a spectral characterization for cubic bipartite graphs with maximum number of perfect matchings.
For general $k$-regular bipartite graphs, it is also interesting to find the spectral characterization for the extremal graphs with maximum number of perfect matchings.

\section*{Acknowledgements}
The research of Ruifang Liu is supported by National Natural Science Foundation of China (Nos. 11971445 and 12171440) and Natural Science Foundation of Henan Province (No. 202300410377). The research of Jie Xue is supported by National Natural Science Foundation of China (No. 12001498)
and China Postdoctoral Science Foundation (No. 2022TQ0303).

\end{document}